\documentclass[a4paper, 11pt]{amsart}
\usepackage[T1]{fontenc}
\usepackage{amssymb,amscd}

\usepackage[all]{xy}

\usepackage[a4paper,hmargin=2.5cm,vmargin=2.5cm]{geometry}
\usepackage{todonotes}
\usepackage{mathrsfs}
\usepackage{mathtools}
\usepackage{graphicx}

\usepackage{xcolor}
\colorlet{NavyBlue}{blue!70!black!90}
\colorlet{PineGreen}{black!50!green}
\usepackage{tikz}

\usepackage{hyperref}
\hypersetup{colorlinks=true, linkcolor=blue, citecolor=black!50!green}

\theoremstyle{plain}
\newtheorem{theorem}{Theorem}[subsection]
\newtheorem{lemma}[theorem]{Lemma}

\newtheorem{prop}[theorem]{Proposition}
\newtheorem{corollary}[theorem]{Corollary}

\newcommand{\cC}{\mathscr{C}}

\newcommand{\cL}{\mathscr{L}}
\newcommand{\cP}{\mathscr{P}}

\newcommand{\RR}{\mathbb{R}}
\newcommand{\ZZ}{\mathbb{Z}}

\DeclareMathOperator{\ch}{ch}
\DeclareMathOperator{\GL}{GL}
\DeclareMathOperator{\id}{id}
\DeclareMathOperator{\rk}{rk}
\DeclareMathOperator{\wt}{wt}
\DeclareMathOperator{\Ker}{Ker}

\theoremstyle{remark}

\newtheorem{remark}[theorem]{Remark}

\theoremstyle{definition}
\newtheorem{example}[theorem]{Example}
\newtheorem{definition}[theorem]{Definition}
\newtheorem{construction}[theorem]{Construction}

\begin{document}

\title{Lascoux polynomials and subdivisions of Gelfand--Zetlin polytopes}
\author{Ekaterina Presnova}
\address{HSE University, Usacheva 6, 119048 Moscow, Russia}
\email{epresnova@hse.ru, ekaterina.presnova@gmail.com}
\author{Evgeny Smirnov}
\address{HSE University, Usacheva 6, 119048 Moscow, Russia}
\address{Independent University of Moscow, Bolshoi Vlassievskii pereulok 11, 119002 Moscow, Russia}
\address{Guangdong Technion -- Israel Institute of Technology,  Daxue road 241, Shantou, Guangdong, 515063 China}
\email{esmirnov@hse.ru, evgeny.smirnov@gmail.com}

\thanks{This work was supported by the HSE University Basic Research Program and by Basis Foundation Fellowship ``Junior Leader''}

\begin{abstract}
    We give a new combinatorial description for Grassmannian Grothendieck polynomials in terms of subdivisions of Gelfand--Zetlin polytopes. Moreover, these subdivisions also provide a description of Lascoux polynomials. This generalizes a similar result on key polynomials by Kiritchenko, Smirnov, and Timorin. 
\end{abstract}

\maketitle

%\tableofcontents 

\section{Introduction}

In this paper, we provide a new combinatorial description of Lascoux polynomials in terms of subdivisions of Gelfand--Zetlin polytopes and certain collections of their faces. Lascoux polynomials, denoted by $\cL_\alpha$, form a basis for $\ZZ[\beta][x_1,x_2,\dots]$, where $\alpha$ runs over the set of weak compositions  (i.e., infinite sequences of nonnegative integers with finitely many positive entries). They simultaneously generalize key polynomials, the characters of Demazure modules, and Grassmannian Grothendieck polynomials; the latter family represents classes of structure sheaves of Schubert varieties in the connective $K$-theory of a Grassmannian, as shown by A.\,Buch~\cite{Buch02}. Both of these families are superfamilies of Schur polynomials.

Lascoux polynomials were defined by A.\,Lascoux~\cite{Lascoux04} in terms of homogeneous divided difference operators; just as other remarkable families of polynomials defined via these operatorsx, they have nonnegative coefficients. Originally, Lascoux defined these  polynomials without variable $\beta$. The introduction of the grading variable $\beta$ is motivated by the connective $K$-theory and goes back to works of S.\,Fomin and An.\,Kirillov \cite{FominKirillov96}, \cite{FominKirillov94}.

Although Lascoux polynomials do not have a description in geometric or representation-theoretic terms, they admit several combinatorial descriptions. V.\,Buciumas, T.\,Scrim\-shaw, and K.\,Weber~\cite{BSW20} establish their connection to the five-vertex model. T.\,Yu~\cite{Yu21} provides a description in terms of set-valued tableaux, simultaneously generalizing Buch's description of Grassmannian Grothendieck polynomials in terms of set-valued Young tableaux and Lascoux--Sch\"utzenberger's tableau description of key polynomials. 

Lascoux polynomials $\cL_\alpha$ specialized at $\beta=0$ are equal to key polynomials. Suppose $w\in S_n$ is a permutation such that $\alpha=(\alpha_1,\dots,\alpha_n)=w(\lambda)$ for a suitable partition $\lambda=(\lambda_1,\dots,\lambda_n)$. The key polynomials $\kappa_\alpha=\kappa_{w,\lambda}$ are defined as the characters of Demazure modules $D_{w,\lambda}$, i.e. $B$-submodules in the irreducible $\GL(n)$-representation $V_\lambda$ with the highest weight $\lambda$. The module $D_{w,\lambda}$ is defined as the smallest $B$-submodule containing the extremal vector $wv_\lambda\in V_\lambda$, where $B\subset \mathrm{GL}(n)$ is a fixed Borel subgroup. Demazure modules were defined by M.\,Demazure in~\cite{Demazure74}; in the same paper the character formula for $D_{w,\lambda}$ was stated. However, as it was pointed out by V.\,Kac, its proof contained a gap; a correct proof appeared more than ten years later, in H.\,H.\,Andersen's work~\cite{Andersen85}. The first combinatorial interpretation of coefficients of key polynomials is due to A.\,Lascoux and M.-P.\,Sch\"utzenberger~\cite{LascouxSchutzenberger90}. Other combinatorial descriptions were obtained by S.\,Mason~\cite{Mason09}; they make use of the fact that key polynomials can be obtained as specializations of non-symmetric Macdonald polynomials, see~\cite{HaglundHaimanLoehr08}.

These days, Lascoux polynomials are becoming a popular object of research in algebraic combinatorics. Among some recent works, let us mention the proof of V.\,Reiner and A.\,Yong's conjecture~\cite{ReinerYong21} on an expansion of Grothendieck polynomials into Lascoux polynomials given by  M.\,Shimozono and T.\,Yu in~\cite{ShimozonoYu23}, and the work by J.\,Pan and T.\,Yu~\cite{PanYu23} on the top-degree components of Lascoux polynomials.

The main goal of this paper is to relate Lascoux polynomials to Gelfand--Zetlin polytopes. These polytopes play a remarkable role in representation theory. They were defined in 1950 in a note~\cite{GelfandZetlin50} by I.\,M.\,Gelfand and M.\,L.\,Zetlin (also spelled Tsetlin or Cetlin), for constructing certain ``nice'' bases in finite-dimensional representations of $\GL(n)$. Namely, vectors of such a basis in a $\GL(n)$-module $V_\lambda$ with the highest weight $\lambda$ are indexed by the integer points in a certain integer convex polytope  $GZ(\lambda)\subset \RR^{n(n-1)/2}$, called the Gelfand--Zetlin polytope.

Gelfand--Zetlin polytopes are also closely related to flag varieties (of type $A$). They are the moment polytopes for the toric degenerations of flag varieties due to N.\,Gonciulea and V.\,Lakshmibai, see~\cite{GonciuleaLakshmibai96}. These degenerations are singular, so the corresponding polytopes are not integrally simple (in fact, even not simple). One can consider the degenerations of Schubert varieties in a flag variety; this was studied by M.\,Kogan and E.\,Miller in~\cite{KoganMiller05}. In this paper the authors explicitly point out a set of faces of Gelfand--Zetlin polytope corresponding to the degeneration of each Schubert variety $X_w$; this degenerated variety can be reducible, so each face corresponds to an irreducible component. Each face corresponds to an rc-graph (pipe dream) for permutation $w\in S_n$. This result is closely related to the monomial degeneration of affine Schubert varieties by A.\,Knutson and E.\,Miller~\cite{KnutsonMiller05}.

In~\cite{KST}, V.\,Kiritchenko, E.\,Smirnov, and V.\,Timorin   generalized this construction to compute products in the cohomology ring of flag variety. They also provided a relation between Schubert varieties, key polynomials and Gelfand--Zetlin polytopes. It is well-known that the Gelfand--Zetlin polytope $GZ(\lambda)$ for a  dominant weight for $\GL(n)$  admits a projection $\pi\colon GZ(\lambda)\to \wt(\lambda)$ to the weight polytope of the $\GL(n)$-module $V_\lambda$ with highest weight $\lambda$. It turns out that if we take Kogan and Miller's collection  of faces $F_{w,\lambda}$ of $GZ(\lambda)$ corresponding to the degeneration of $X_w$, then the key polynomial can be obtained as $\kappa_{w,\lambda}=\sum \exp(\pi(z))$, where $z$ ranges over the set of integer points in $F_{w,\lambda}$ (see \cite[Corollary~5.2]{KST}).

In this paper we generalize this result to the case of Grassmannian Grothendieck and Lascoux polynomials, constructing their combinatorial description in terms of subdivisions of Gelfand--Zetlin polytopes. For this we construct a cellular decomposition $\cC$ of $GZ(\lambda)$ whose 0-cells coincide with  integer points in $GZ(\lambda)$. Now, to each $i$-dimensional cell $C_i$ we assign a monomial $m(C_i)$ in $x_1,\dots,x_n$; for a 0-cell $z\in GZ(\lambda)$ we have $m(C_i)=\exp(\pi(z))$. Some cells correspond to the zero monomial. Our main result is as follows:
\[
\cL_{w,\lambda}=\sum_{C_i\in \cC\cap F_{w,\lambda}} \beta^i m(C_i),
\]
where the sum is taken over all cells situated inside the collection of faces $F_{w,\lambda}$.

Informally, the Lascoux polynomial $\cL_{w,\lambda}$ can be viewed as a ``weighted Euler characteristic'' of the subdivision $\cC\cap F_{w,\lambda}$ for the collection of faces $F_{w,\lambda}$. Namely, $i$-dimensional cells of this subdivision correspond to monomials of degree $i+\ell(w)$ with coefficient $\beta^i$ in front of them.

It would be very interesting to establish a bijection of our construction of cells indexing monomials in Lascoux polynomials with T.\,Yu's description in terms of set-valued tableaux. In particular, we expect the crystal operations on set-valued tableaux (see~\cite{Yu21}) to have a nice description in terms of Gelfand--Zetlin polynomials. However, we do not address these questions in this paper, leaving them as a subject of subsequent work. 

\subsection*{Structure of the paper} This paper is organized as follows. We recall the definitions of Lascoux polynomials and Gelfand--Zetlin polytopes in Section~\ref{sec:prelim}. In Section~\ref{sec:example} we consider the first interesting example, describing the construction for $n = 3$. The main result is stated in Section~\ref{sec:main}. Its proof is given in Section~\ref{sec:proofs-w0} for the case of the longest permutation $w=w_0$ and in Section~\ref{sec:proofgen} for an arbitrary permutation $w$, respectively. 

\subsection*{Acknowledgements} We are grateful to Valentina Kiritchenko for useful discussions. We also would like to thank the anonymous referee for numerous valuable comments and suggestions that significantly improved the exposition. 

\section{Preliminaries}\label{sec:prelim}

 In the first part of this section we begin with definitions of divided difference operators and Demazure--Lascoux operators and their properties. Next, we define the Lascoux polynomials and state how they are related to key polynomials, Grassmannian Grothendieck polynomials, and Schur polynomials. In the second part of this section we will describe the Gelfand--Zetlin polytopes and their faces.

\subsection{Lascoux polynomials}\label{ssec:lascoux}

To define Lascoux polynomials, we need two families of operators: \emph{divided difference operators} $\partial_i$, with $1\leq i\leq n-1$, acting on the polynomial ring $\ZZ[x_1,\dots,x_n]$, and \emph{Demazure--Lascoux operators} $\pi_i^{(\beta)}$, again with $1\leq i\leq n-1$, acting on the ring $\ZZ[\beta,x_1,\dots,x_n]$ equipped with a formal parameter $\beta$. 

The parameter $\beta$ appears in the connective $K$-theory of a Grassmannian; taking $\beta=-1$ and $\beta=0$, we recover the usual $K$-theory and the cohomology ring of a Grassmannian respectively.

\begin{definition}\label{def:divdiff}
The $i^{th}$ \emph{divided difference operator} $\partial_i$ acts on polynomial $f = f(x_1, x_2, \ldots)$ in the following way:
 $$
 \partial_i (f) = \frac{f - s_if}{x_i - x_{i+1}},
 $$
where $s_if$ is obtained from $f$ by permuting variables $x_i$ and $x_{i+1}$. 
\end{definition}

\begin{definition}\label{def:demlascoux}
    \textit{The $i^{th}$ Demazure--Lascoux operator} $\pi_i^{(\beta)}$ acts on polynomial $f \in \mathbb{Z}[\beta][x_1, x_2, \ldots]$ in the following way:
$$
 \pi_i^{(\beta)}(f) = \partial_i(x_i f + \beta x_ix_{i+1}f ).
 $$
\end{definition}

The following properties of Demazure--Lascoux operators are immediate.

\begin{prop}\label{prop:lascoux}
 Demazure--Lascoux operators $\pi_i^{(\beta)}$ are idempotent linear operators satisfying the braid relations. Namely:
 \begin{itemize}
     \item If $f = s_if$, then $\pi_i^{(\beta)}(f) = f$;
     \item $(\pi_i^{(\beta)})^2 = \pi_i^{(\beta)}$;
     \item $\pi_i^{(\beta)} \pi_j^{(\beta)} = \pi_j^{(\beta)}\pi_i^{(\beta)}$ if $|i - j| > 1$;
     \item  $\pi_i^{(\beta)}\pi_{i+1}^{(\beta)}\pi_i^{(\beta)} = \pi_{i+1}^{(\beta)}\pi_i^{(\beta)}\pi_{i+1}^{(\beta)}$.
 \end{itemize}
\end{prop}

\begin{proof}
Straightforward computation. 
\end{proof}

 Let $\alpha = (\alpha_1, \alpha_2, \ldots)$ be an infinite sequence of nonnegative integers with finitely many positive entries.
 \begin{definition}\label{def:lascoux} \textit{The Lascoux polynomial $\mathscr{L}_{\alpha} \in \mathbb{Z}[\beta][x_1, x_2, \ldots]$ associated with $\alpha$} is defined by:
 \begin{equation*}
\mathscr{L}_{\alpha} = 
 \begin{cases}
   x^{\alpha} &\text{if $\alpha$ is a partition: $\alpha_1 \geq \alpha_{2} \geq \ldots$}\\
   \pi_i^{(\beta)} (\mathscr{L}_{s_i \alpha})&\text{otherwise, where $\alpha_i < \alpha_{i + 1}$}
 \end{cases}
\end{equation*}
\end{definition}
 Since the Demazure--Lascoux operators satisfy the braid relations, we can associate a Lascoux polynomial to partition $\lambda$ and permutation $w \in S_n$ in the following way:
 $$
 \mathscr{L}_{w, \lambda} = \pi_{i_k}^{(\beta)}\ldots\pi_{i_2}^{(\beta)}\pi_{i_1}^{(\beta)}(x^{\lambda}),
 $$
where $(s_{i_k},\ldots, s_{i_1})$ is a reduced word for permutation $w=s_{i_1}\ldots s_{i_k}$.

It is well-known (cf., for instance, \cite{Yu21}) that specializations of Lascoux polynomials provide other nice families of polynomials.

\begin{theorem}\label{thm:4families}
\begin{enumerate}
    \item Key polynomials are obtained by specializing the Lascoux polynomials at~$\beta = 0$:
$$
\kappa_{w, \lambda} = \mathscr{L}_{w, \lambda}\mid_{\beta = 0};
$$

\item Grassmannian Grothendieck polynomials are equal to Lascoux polynomials with permutation~$w_0$:
$$
G_{\lambda}^{(\beta)} = \mathscr{L}_{w_0, \lambda} = \pi_{w_0}^{(\beta)}(x^\lambda);
$$

\item Schur polynomials are equal to key polynomials for permutation~$w_0$, or, equivalently, to Grassmannian Grothendieck polynomials for $\beta=0$:
$$
S_{\lambda}=\kappa_{w_0, \lambda} = \pi_{w_0}^{(\beta)}(x^\lambda)|_{\beta = 0}.
$$
\end{enumerate}
\end{theorem}

\subsection{Gelfand--Zetlin polytopes and Gelfand--Zetlin patterns}\label{ssec:gz}

Let $\lambda$ be a partition, i.e. a sequence of nonnegative integers $ \lambda_1 \geq \lambda_2 \geq \ldots \geq \lambda_n$. Consider the space $\RR^d$, where $d = \frac{n (n-1)}{2}$, with coordinates $y_{ij}$ indexed by pairs $(i, j)$ of positive integers satisfying $i + j \leq n$. Consider the system of inequalities defined by the following tableau:
\begin{equation}\label{eq:gz-tableau}
	\begin{matrix}
		\lambda_n && \lambda_{n-1}&& \lambda_{n-2}&&\dots &&\lambda_1\\
		& y_{11} && y_{12} && \dots && y_{1,n-1}\\
		&& y_{21}&& \dots && y_{2,n-2}\\
		&&& \ddots & \vdots &\reflectbox{$\ddots$}\\
		&&&&y_{n-1,1}
	\end{matrix},
\end{equation}
where every triple of variables $a, b, c$ in each small triangle $\begin{matrix} a && b \\ & c \end{matrix}$ satisfies the inequalities $a \leq c \leq b$.

\begin{definition} A \emph{Gelfand--Zetlin polytope} $GZ(\lambda)$ is the set of points in $\RR^{n(n-1)/2}$ satisfying the set of inequalities defined by~(\ref{eq:gz-tableau}). A \emph{Gelfand--Zetlin pattern} is a tableau of \emph{integer} coordinates $y_{ij}$ satisfying the same inequalities.  In other words, a Gelfand--Zetlin pattern is the set of coordinates of an integer point in $GZ(\lambda)$. 
\end{definition}

\begin{remark}\label{rem:shifts} Obviously, adding the same integer $k$ to all $\lambda_i$ (that is, replacing the first row of the tableau by $(\lambda_n+k,\lambda_{n-1}+k,\dots,\lambda_1+k)$) results in shifting every coordinate of $GZ(\lambda)$ by~$k$. So, up to a parallel translation $GZ(\lambda)$ is defined not by $\lambda_i$, but rather by their differences $\lambda_i-\lambda_{i-1}$. Further in the examples we will sometimes assume $\lambda_n=0$. Also note that if all $\lambda_i-\lambda_{i-1}$ are nonzero, $GZ(\lambda)$ is full-dimensional; moreover, all such polytopes have the same normal fan.
\end{remark}

The following theorem is classical.

\begin{theorem}[\cite{GelfandZetlin50}]\label{thm:gz50} The number of Gelfand--Zetlin patterns is equal to the dimension of $\GL(n)$-module $V_{\lambda}$ with the highest weight $\lambda$. It can be computed using Weyl's dimension formula:
\[
\#(GZ(\lambda)\cap \ZZ^d)=\dim V_{\lambda}=\prod_{1\leq i<j\leq n} \frac{\lambda_i-\lambda_j-i+j}{j-i}.
\]
\end{theorem}

Gelfand--Zetlin patterns parametrize elements of a certain basis in $V_{\lambda}$, constructed as follows. Consider the upper-left corner subgroup $\GL(n-1)\subset \GL(n)$ and restrict $V_{\lambda}$ to this subgroup. As a $\GL(n-1)$-module, it will be reducible, but multiplicity free, meaning that every irreducible component occurs at most once: $V_{\lambda}=\bigoplus_\mu V_\mu$. The weights $\mu=(\mu_1,\dots,\mu_{n-1})$ are highest weights of  $\GL(n-1)$-modules; they satisfy the intermittence condition: $\lambda_1\geq \mu_1\geq\lambda_2\geq\dots \geq \mu_{n-1}\geq \lambda_n$. Write the components $\mu_i$ from right to left as the second row of a Gelfand--Zetlin patterns. Now restrict each of $\GL(n-1)$-modules $V_{\mu_i}$ to the subgroup $\GL(n-2)$, etc.; at the end we obtain a set of one-dimensional subspaces of $V_{\lambda}$ (representations of $\GL(1)$). Picking a nonzero vector in each of these subspaces, we obtain a \emph{Gelfand--Zetlin basis} of $V_{\lambda}$ indexed by the sets of intermitting highest weights, i.e., by Gelfand--Zetlin patterns.

\subsection{Faces of Gelfand--Zetlin polytopes}

Faces of a Gelfand--Zetlin polytope are obtained by replacing some of the defining inequalities by equalities. Following~\cite{KST}, we will represent them by \emph{face diagrams}. An example of face diagram is given in Figure~\ref{fig:fig2}. Dots in this diagram correspond to coordinates $y_{ij}$, while an edge between $y_{ij}$ and $y_{i-1 j}$ means that in the system of inequalities, $y_{i-1j} \leq y_{ij}$ is replaced by $y_{i-1j} = y_{ij}$. The same happens with edges between $y_{ij}$ and $y_{i-1j + 1}$. Here we formally set $y_{0i}=\lambda_{n+1-i}$, so some variables may be equal to entries $\lambda_i$ from the top row.

\begin{example}\label{ex:21}
Let $\lambda = (\lambda_1,\lambda_2,0)$. The polytope $GZ(\lambda)$ is shown in Figure~\ref{fig:fig1}, the diagram corresponding to the shaded face is shown in Figure~\ref{fig:fig2}.

\begin{figure}[h!]
\centering
\begin{minipage}[b]{0.4\textwidth}
\centering
\begin{tikzpicture}% [scale = 0.9]
\filldraw[black!10] (2,2) -- (3.5, 4.5) -- (3.5,2.5);
\draw [gray] (1.5, 4.5)  -- (1.5,0.5);
\draw (0,0) -- (0,2) -- (2,2) -- (0,0) -- (1.5,0.5) -- (3.5, 2.5) -- (3.5, 4.5) -- (2,2) -- (3.5,2.5) -- (3.5, 4.5)-- (1.5, 4.5) -- (0,2);
\draw (0,1) node[left]{$\lambda_1-\lambda_2$};
\draw (2.5,4.5) node[above]{$\lambda_1-\lambda_2$};
\draw (0.85, 0.2) node[below]{$\lambda_2-\lambda_3$};
\draw (3.5,3.5) node[right]{$\lambda_2-\lambda_3$};
\end{tikzpicture}
\caption{Gelfand--Zetlin polytope}\label{fig:fig1}
\end{minipage}
\hfill
\begin{minipage}[b]{0.5\textwidth}
\centering
\begin{tikzpicture}
\draw (2,3.5) -- (1,2.5);
\filldraw (0,3.5) circle (2pt);
\filldraw (2,3.5) circle (2pt);
\filldraw (4,3.5) circle (2pt);
\filldraw (1,2.5) circle (2pt);
\filldraw (3,2.5) circle (2pt);
\filldraw (2,1.5) circle (2pt);
\end{tikzpicture}
\caption{Face diagram of the shaded face}\label{fig:fig2}
\end{minipage}
\end{figure}

\end{example}

Note that these equalities are in general not independent: for each four vertices forming a ``diamond'' in three consecutive rows, the top two equalities imply the two bottom ones, and vice versa. Indeed, for a ``diamond'' $\begin{smallmatrix} & y_{i-1,j}\\ y_{i,j-1} && y_{ij}\\ &y_{i+1, j-1}\end{smallmatrix}$, the equalities $y_{i,j-1}=y_{i-1,j}=y_{i,j}$ imply that $y_{i+1,j-1}$ is also equal to these three values, since both $y_{i+1,j-1}$ and $y_{i-1,j}$ are ``squeezed'' between $y_{i,j-1}$ and $y_{ij}$; the same holds for equalities in the bottom row.

This means that Gelfand--Zetlin polytopes are not simple (in fact, they are ``highly non-simple''): the intersection of certain facets can have codimension strictly less than the number of facets.

\subsection{Dual Kogan faces}\label{ssec:dualkogan} These are faces of Gelfand--Zetlin polytopes of some special form.

\begin{definition}\label{def:dualkogan}
	A face of Gelfand--Zetlin polytope is called a \emph{dual Kogan face} if it is defined only by equations of the form $y_{ij}=y_{i+1,j-1}$ for $i\geq 0$.
\end{definition}

Equivalently, dual Kogan faces are exactly those containing the ``maximal'' vertex defined by the equations $\lambda_1=y_{1,n-1}=y_{2,n-2}=\dots=y_{n-1,1}$, $\lambda_2=y_{1,n-2}=\dots=y_{n-2,1}$, and so on. Note that the ``maximal'' dual Kogan vertex is simple. We also formally consider the whole polytope $GZ(\lambda)$ as a dual Kogan face, defined by the empty set of equations. 

Now we will assign a permutation to each dual Kogan face. Consider the following reduced word for the longest permutation $w_0\in S_n$: 
\[
\mathbf{w}_0=(s_{n-1},\dots,s_{1}, s_{n-2},\dots,s_{1},\dots, s_1,s_2,s_1)
\]

To each dual Kogan face $F$ we assign a subword $\mathbf{w}^-(F)$ of $\mathbf w_0$ as follows. Consider the face diagram and write simple transposition $s_{n-j}$ on all edges corresponding to the equalities $y_{i,j+1}=y_{i+1,j}$ (see Figure~\ref{fig:dualkogan}).

\begin{figure}[h!]
\begin{tikzpicture} [scale = 1.5]
    \filldraw (0,6) circle (1.5 pt);
    \filldraw (1,6) circle (1.5 pt);
    \filldraw (2,6) circle (1.5 pt);
    \filldraw (3,6) circle (1.5 pt);
    \filldraw (0.5,5.5) circle (1.5 pt);
    \filldraw (1.5,5.5) circle (1.5 pt);
    \filldraw (2.5,5.5) circle (1.5 pt);
    \filldraw (1,5) circle (1.5 pt);
    \filldraw (2,5) circle (1.5 pt);
    \filldraw (1.5,4.5) circle (1.5 pt);
    \draw[lightgray] (1.5, 4.5) -- (3,6);
    \draw[lightgray] (1,5) -- (2,6);
    \draw[lightgray] (0.5,5.5) -- (1,6);
    \draw (0.75,5.75) node {$s_3$};
    \draw (1.75,5.75) node {$s_2$};
    \draw (2.75,5.75) node {$s_1$};
    \draw (1.25,5.25) node {$s_3$};
    \draw (2.25,5.25) node {$s_2$};
    \draw (1.75,4.75) node {$s_3$};
\end{tikzpicture}
\caption{Assigning word to a dual Kogan face}\label{fig:dualkogan}
\end{figure}
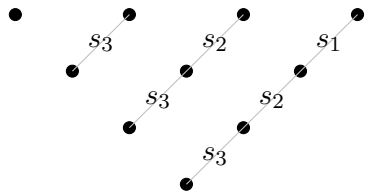

Now, for a face $F$, the corresponding word $\mathbf{w}^-(F)$ is obtained by reading the face diagram from bottom to top, from right to left, and taking the simple transpositions corresponding to the equations. Face $F$ is said to be \emph{reduced}, if the word $\mathbf{w}^-(F)$ is reduced, and \emph{non-reduced} in the opposite case. For a reduced face, we denote $w(F)=w_0w^-(F)$, where $w^-(F)$ is the permutation obtained by taking the product of simple transpositions in $\mathbf{w}^-(F)$. Figure\,\ref{fig:gz3-dualfaces} below shows all the face diagrams for reduced dual Kogan faces in a three-dimensional Gelfand--Zetlin polytope and the corresponding permutations $w(F)$. Note that the dimension of $F$ is equal to the length of $w(F)$.

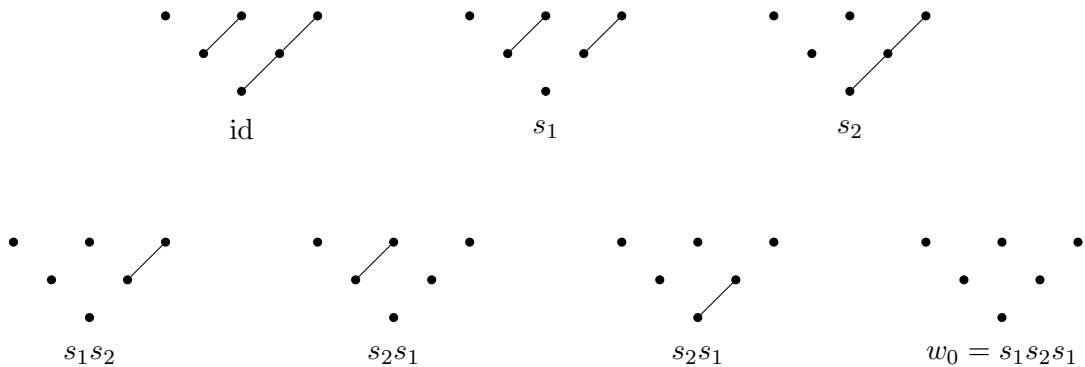
\begin{figure}[h!]
$$
\begin{tikzpicture} [scale = 1]
    \filldraw (0,6) circle (1.5 pt);
    \filldraw (1,6) circle (1.5 pt);
    \filldraw (2,6) circle (1.5 pt);
    \filldraw (0.5,5.5) circle (1.5 pt);
    \filldraw (1.5,5.5) circle (1.5 pt);
    \filldraw (1,5) circle (1.5 pt);
    \draw (1, 5) -- (2,6);
    \draw (0.5,5.5) -- (1,6);
    \draw (1, 4.5) node {$\mathrm{id}$};

    \filldraw (4,6) circle (1.5 pt);
    \filldraw (5,6) circle (1.5 pt);
    \filldraw (6,6) circle (1.5 pt);
    \filldraw (4.5,5.5) circle (1.5 pt);
    \filldraw (5.5,5.5) circle (1.5 pt);
    \filldraw (5,5) circle (1.5 pt);
    \draw (4.5,5.5) -- (5,6);
    \draw (5.5,5.5) -- (6,6); 
    \draw (5,4.5) node {$s_1$};

    \filldraw (8,6) circle (1.5 pt);
    \filldraw (9,6) circle (1.5 pt);
    \filldraw (10,6) circle (1.5 pt);
    \filldraw (8.5,5.5) circle (1.5 pt);
    \filldraw (9.5,5.5) circle (1.5 pt);
    \filldraw (9,5) circle (1.5 pt);
    \draw (9,5) -- (10,6);
    \draw (9,4.5) node {$s_2$};

    \filldraw (-2,3) circle (1.5 pt);
    \filldraw (-1,3) circle (1.5 pt);
    \filldraw (0,3) circle (1.5 pt);
    \filldraw (-0.5,2.5) circle (1.5 pt);
    \filldraw (-1.5,2.5) circle (1.5 pt);
    \filldraw (-1,2) circle (1.5 pt);
    \draw (-0.5,2.5) -- (0,3);
    \draw (-1,1.5) node {$s_1s_2$};

    \filldraw (2,3) circle (1.5 pt);
    \filldraw (3,3) circle (1.5 pt);
    \filldraw (4,3) circle (1.5 pt);
    \filldraw (2.5,2.5) circle (1.5 pt);
    \filldraw (3.5,2.5) circle (1.5 pt);
    \filldraw (3,2) circle (1.5 pt);
    \draw (2.5,2.5) -- (3,3);
    \draw (3,1.5) node {$s_2s_1$};

    \filldraw (6,3) circle (1.5 pt);
    \filldraw (7,3) circle (1.5 pt);
    \filldraw (8,3) circle (1.5 pt);
    \filldraw (6.5,2.5) circle (1.5 pt);
    \filldraw (7.5,2.5) circle (1.5 pt);
    \filldraw (7,2) circle (1.5 pt);
    \draw (7,2) -- (7.5,2.5);
    \draw (7,1.5) node {$s_2s_1$};

    \filldraw (10,3) circle (1.5 pt);
    \filldraw (11,3) circle (1.5 pt);
    \filldraw (12,3) circle (1.5 pt);
    \filldraw (10.5,2.5) circle (1.5 pt);
    \filldraw (11.5,2.5) circle (1.5 pt);
    \filldraw (11,2) circle (1.5 pt);
    \draw (11,1.5) node {$w_0 = s_1s_2s_1$};
\end{tikzpicture}
$$
\caption{Reduced dual Kogan faces and permutations for $n=3$.}\label{fig:gz3-dualfaces}
\end{figure}

\begin{remark}
In~\cite{KnutsonMiller05} and~\cite{KnutsonMiller04}, A.\,Knutson and E.\,Miller define \emph{subword complexes} and, more specifically, \emph{pipe dream complexes}. It readily follows from the definitions that as a CW-complex the link of the ``maximal'' dual Kogan vertex is nothing but the subword complex corresponding to the word $\mathbf{w}_0$ of the longest permutation $w_0$. Each dual Kogan face $F$ of $GZ(\lambda)$ corresponds to a pipe dream with permutation $w^-(F)$. To construct the pipe dream from a face diagram, for each equality $y_{ij}=y_{i+1,j-1}$, put a cross in the box $(n-j+1,i+1)$, and elbows in all the remaining boxes. 
\end{remark}

\subsection{Key polynomials} The following theorem is due to V.\,Kiritchenko, E.\,Smirnov, and V.\,Timorin~\cite{KST}. It provides a description of key polynomials in terms of integer points in dual Kogan faces. Here we state it in a different (but equivalent) way: the authors of the original paper used Kogan faces, as opposed to dual Kogan faces that we are using, and lowest-weight modules instead of highest-weight ones.

Before we proceed, define the following projection map $\pi\colon\RR^\frac{n(n-1)}2\to \RR^n$ from $GZ(\lambda)$ to the weight polytope $\wt(\lambda)$ of the $\GL(n)$-module $V_{\lambda}$:
\[
\pi\colon	\begin{pmatrix}
		 y_{11} && y_{12} && \dots && y_{1,n-1}\\
		& y_{21}&& \dots && y_{2,n-2}\\
		&& \ddots & \vdots &\reflectbox{$\ddots$}\\
		&&&y_{n-1,1}
	\end{pmatrix}\mapsto 
	\begin{pmatrix} y_{n-1,1}\\
	(y_{n-2,1}+y_{n-2,2})-y_{n-1,1}\\
	\vdots\\
	(y_{11}+\dots+y_{1,n-1})-(y_{21}+\dots+y_{2,n-2})
	\end{pmatrix}.
\]
This map takes a Gelfand--Zetlin pattern into the vector with the $i$-th component equal to the difference of sums of entries in $i$-th and $(i+1)$-th rows of the pattern, \emph{counted from below}.  

\begin{definition} For an integer point $z\in GZ(\lambda)\subset \RR^\frac{n(n-1)}{2}$, define its \emph{character} $\ch z\in\ZZ[x_1,\dots,x_n]$ as follows: $\ch z=x_1^{a_1}\dots x_n^{a_n}$, where $(a_1,\dots,a_n)=\pi(z)$. More generally, for an arbitrary subset~$S\subset\RR^\frac{n(n-1)}{2}$, define its \emph{character} $\ch S\in\ZZ[x_1,\dots,x_{n}]$ as the sum of all monomials~$\ch z$ for all integer points~$z\in S\cap \ZZ^{\frac{n(n-1)}2}$.
\end{definition}

\begin{theorem}[{\cite[Theorem 5.1]{KST}}]\label{thm:key-kst} The key polynomial $\kappa_{w,\lambda}$ is equal to the character of the union of the dual Kogan faces in $GZ(\lambda)$ corresponding to the permutation $w \in S_n$:
$$
 \kappa_{w,\lambda} = \mathrm{ch}\left( \bigcup \limits_{w(F) = w} {F} \right).
$$
\end{theorem}

\begin{example} Let $S=GZ(\lambda)$ be the whole Gelfand--Zetlin polytope. Then, according to Theorem~\ref{thm:gz50}, its character $\ch S$ is nothing but the character of representation~$V_{\lambda}$, i.e., the Schur polynomial $S_\lambda(x_1,\dots,x_n)$.
\end{example}

\section{The three-dimensional case}\label{sec:example}

In this section we study the three-dimensional case. We start with describing a cellular decomposition for $GZ(\lambda)$, where $\lambda = (3, 2, 0)$, and show how monomials of Lascoux polynomials $\mathscr{L}_{w, \lambda}$ correspond to cells of this decomposition. Next, we will describe the construction in general for arbitrary $GZ(\lambda_1, \lambda_2, \lambda_3)$. The construction presented in this section is \emph{ad hoc}; we provide the cell decomposition in the general case in the subsequent section.

\subsection{Example}\label{ssec:ex320} Let $\lambda=(3, 2, 0)$. 
The Gelfand--Zetlin polytope $GZ(3, 2, 0)$ is given by the following tableau:
\[
\begin{matrix}
0 && 2 && 3 \\
&x&&y\\
&&z\\
\end{matrix}
\]

Our goal is to construct a cellular decomposition of $GZ(\lambda)$ and assign to each cell a monomial in $\beta,x_1,\dots,x_n$  with the following properties: first, each face of $GZ(\lambda)$ should be the union of some cells. Second, for the set of dual Kogan faces corresponding to a permutation $w\in S_n$, the sum of monomials corresponding to the cells forming these faces should be equal to the Lascoux polynomial $\cL_{w,\lambda}$.

We proceed by induction on the length of $w$. First put the monomial $x_1^3x_2^2$ into the vertex corresponding to the identity permutation. Then take the integer points on one-dimensional dual Kogan faces, thus getting their subdivision. Then we assign monomials occurring in $\pi_1^{(\beta)}(x_1^3x_2^2)$ and $\pi_2^{(\beta)}(x_1^3x_2^2)$ to vertices and segments of corresponding one-dimensional faces; see Figure~\ref{fig:ids1s2}.

\begin{figure}[h!]
\begin{tikzpicture} [scale = 1]

\filldraw [black] (0,0) circle (3pt);
\draw (0,0) node [above]{$ x_1^3 x_2^2$};
\draw (0, -2) node {\text{$\mathscr{L}_{\id, \lambda}$}};

\draw (4,-1) -- (4,1);
\filldraw [gray] (4,-1) circle (3pt);
\filldraw [black] (4,1) circle (3pt);
\draw (4,-1) node [right]{$ x_1^2 x_2^3$};
\draw (4,1) node [right]{$ x_1^3 x_2^2$};
\draw (4,0) node [right]{$ \beta x_1^3 x_2^3$};
\draw (4, -2) node {\text{$\mathscr{L}_{s_1, \lambda}$}};

\draw (8,0) node[above] {$ x_1^3 x_3^2$}-- node[below] {$ \beta x_1^3 x_2 x_3^2$}(10,0) node[above] {$ x_1^3 x_2 x_3$}-- node[below] {$\beta x_1^3 x_2^2x_3$}(12, 0)node[above] {$ x_1^3 x_2^2$};
\filldraw [gray] (8,0) circle (3pt);
\filldraw [gray] (10,0) circle (3pt);
\filldraw [black] (12,0) circle (3pt);
\draw (10, -2) node {\text{$\mathscr{L}_{s_2, \lambda}$}};

\end{tikzpicture}
\caption{Lascoux polynomials for $\id$, $s_1$, and $s_2$}\label{fig:ids1s2}
\end{figure}
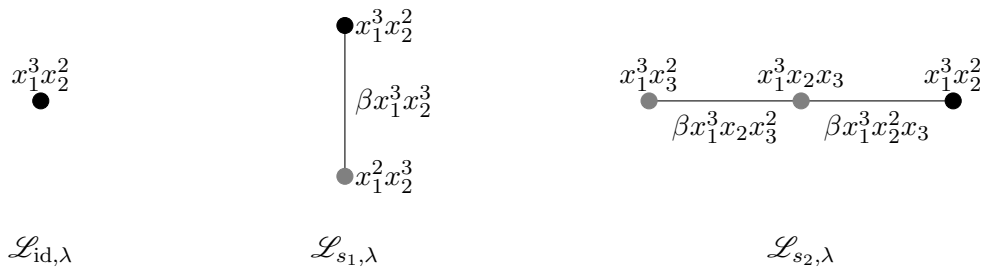

Acting again with $\pi_1^{(\beta)}$ and $\pi_2^{(\beta)}$ respectively, we obtain the subdivisions for two-dimensional faces shown on Figure~\ref{fig:s1s2s2s1}.

\begin{figure}[h!]
\begin{tikzpicture}[scale=1]

\filldraw [blue!5] (10.75, 5.25) -- (10.75, 8.25) -- (12.75, 8.25) -- (12.75, 7.25);
\filldraw [teal!10] (10.75, 8.25) -- (10,7) -- (10.75, 7);

\filldraw [teal!95!blue!15](10.75, 8.25) -- (10.75, 7) -- (12,7) -- (12.75, 7.25)  -- (12.75, 8.25);

\draw [NavyBlue!50!,  dashed] (11.75, 6.25) -- (11.75, 8.25) ;
\draw [NavyBlue!50!,  dashed] (10.75, 7.25) -- (11.75, 7.25);
\draw [PineGreen,dotted, very thick] (11.75, 8.25) -- (11, 7);

\draw (10,5) -- (10, 7) -- (12,7) -- (10,5) -- (10.75, 5.25) -- (12.75, 7.25) -- (12,7) -- (12.75, 7.25) -- (12.75, 8.25) -- (10.75, 8.25)  -- (10.75, 5.25) -- (10,5) -- (10, 7) -- (10.75, 8.25) -- (12.75, 8.25)-- (12,7);

\filldraw[black] (12.75, 8.25) circle (2pt);

\scriptsize

\filldraw [blue!5] (0,0) -- (0,6) -- (4,6) -- (4,4);

\draw[dashed, gray, ->] (0, 7) node [above][black] {$\beta^2 x_1^3 x_2^2 x_3^2$} -- (1, 5);
\draw[dashed, gray, ->] (4, 7) node [above][black] {$\beta^2 x_1^3 x_2^3 x_3$} -- (3, 5);
\draw[dashed, gray, ->] (3, 0) node [right][black] {$\beta^2 x_1^2 x_2^3 x_3^2$} -- (1, 2);

\draw (0,6) node [left = 2pt] {$x_1^3x_3^2$}--node [left  = 2pt] {$\beta x_1^3x_2x_3^2$}
    (0,4) node [left  = 2pt] {$x_1^2x_2x_3^2$}--node [left  = 2pt] {$\beta x_1^2x_2^2x_3^2$}
    (0,2) node [left  = 2pt] {$x_1x_2^2x_3^2$}--node [left = 2pt] {$\beta x_1x_2^3x_3^2$}
    (0,0) node [left  = 2pt] {$x_2^3x_3^2$};

\draw (2,6) node [above] {$x_1^3x_2x_3$}--node [above, sloped] {$\beta x_1^3x_2^2x_3$}
    (2,4) node [right = 2pt] {$x_1^2x_2^2x_3$}--node[above, sloped]{$\beta x_1^2x_2^3x_3$}
    (2,2) node [below = 2pt, rotate = 45] {$x_1x_2^3x_3$};

\draw (4,6) node [right = 2pt, NavyBlue] {$x_1^3x_2^2$}--node [right = 2pt] {$\beta x_1^3x_2^3$}
    (4,4) node [right = 2pt] {$x_1^2x_2^3$};

\draw (0, 0) -- node[below, sloped]{$\beta x_1x_2^3x_3^2$} (2, 2) -- node[below, sloped]{$\beta x_1^2x_2^3x_3$} (4, 4);

\draw (0, 4) -- node [below] {$\beta x_1^2x_2^2x_3^2$} (2, 4);

\draw (4, 6) -- node [below = 2pt] {$\beta x_1^3x_2^2x_3$} 
      (2, 6) -- node [below = 2pt] {$\beta x_1^3x_2x_3^2$}
      (0,6);

\filldraw [NavyBlue](4,6)  circle (3pt);
\filldraw [gray] (4,4) circle (3pt);
\filldraw [NavyBlue](2,6) circle (3pt);
\filldraw [gray] (2,4) circle (3pt);
\filldraw [gray] (2,2) circle (3pt);
\filldraw [NavyBlue] (0,6) circle (3pt);
\filldraw [gray] (0,2) circle (3pt);
\filldraw [gray] (0,4) circle (3pt);
\filldraw [gray] (0,0) circle (3pt);
\draw [NavyBlue][very thick] (0,6) -- (4,6);

\filldraw [teal!10] (7,0) -- (9,2) -- (13,2) -- (13, 0);
\draw[dashed, gray, ->] (9, 3.5) node [above][black] {$\beta^2 x_1^3 x_2^2x_3^2$} -- (11, 1);
\draw[dashed, gray, ->] (12, 3.5) node [above][black] {$\beta^2 x_1^3 x_2^3 x_3$} -- (12.5, 1);
\draw[dashed, gray, ->] (6.5,2) node [left][black] {$\beta^2 x_1^3 x_2 x_3^3$} -- (9, 1);

\draw (13, 0) node [below] {$x_1^2x_2^3$} -- node [above] {$\beta x_1^2x_2^3x_3$} 
      (11, 0) node [below] {$x_1^2x_2^2x_3$} -- node [above] {$\beta x_1^2x_2^2x_3^2$}
      (9, 0) node [below] {$x_1^2x_2x_3^2$} -- node [above] {$\beta x_1^2x_2x_3^3$} 
      (7, 0) node [below] {$x_1^2x_3^3$};

\draw (13, 2) node [above = 2pt, PineGreen] {$x_1^3x_2^2$} -- node [below] {$\beta x_1^3x_2^2x_3$} 
      (11, 2) node [above = 2pt ]  {$x_1^3x_2x_3$} -- node [below] {$\beta x_1^3x_2x_3^2$}
      (9, 2) node [above = 2pt ] {$x_1^3x_3^2$};

\draw (13, 2) -- node [above, sloped, pos=0.6] {$\beta x_1^3x_2^2x_3$} (11, 0);
\draw (11, 2) -- node [above, sloped, pos=0.6] {$\beta x_1^3x_2x_3^2$} (9, 0);
\draw (9, 2) -- node [above, sloped, pos=0.6] {$\beta x_1^3x_3^3$} (7, 0);
\draw (13, 2) -- node [right]{$\beta x_1^3x_2^3$}(13, 0);

\filldraw [PineGreen] (13,0) circle (3pt);
\filldraw [gray] (9,0) circle (3pt);
\filldraw [gray] (11,0) circle (3pt);
\filldraw [PineGreen] (13,2) circle (3pt);
\filldraw [gray] (9,2) circle (3pt);
\filldraw [gray] (11,2) circle (3pt);
\filldraw [gray] (7,0) circle (3pt);
\draw [PineGreen] [very thick] (13,0) -- (13, 2);

\normalsize

\draw (11, -2) node {\text{$\mathscr{L}_{s_2s_1, \lambda} = \pi_2^{(\beta)} ( \mathscr{L}_{s_1})$}};

\draw (2, -2) node {\text{$\mathscr{L}_{s_1s_2, \lambda}  = \pi_1^{(\beta)} ( \mathscr{L}_{s_2})$}};
\end{tikzpicture}
\caption{Lascoux polynomials for $s_1s_2$ and $s_2s_1$}\label{fig:s1s2s2s1}
\end{figure}
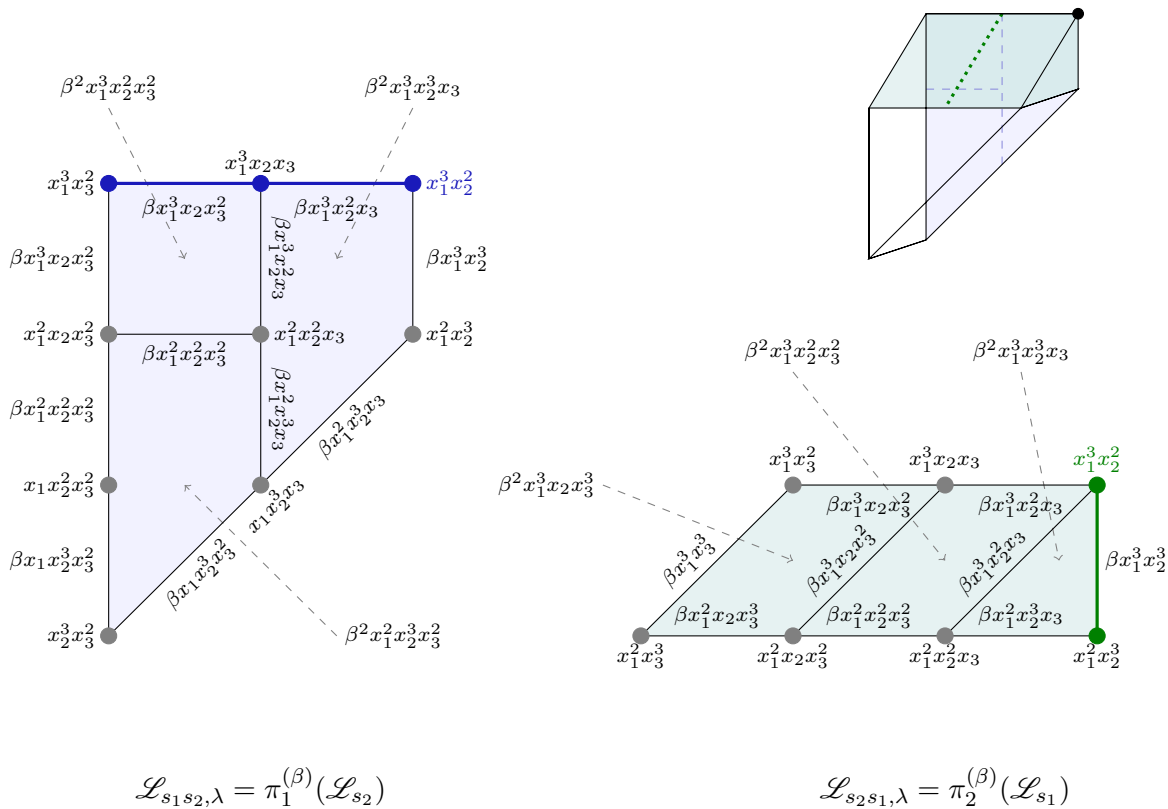

Finally, acting by $\pi_1^{(\beta)}$ on $\mathscr{L}_{s_2s_1, \lambda}$, we get $\mathscr{L}_{s_1s_2s_1, \lambda}$, with monomials corresponding to the cells in the cellular decomposition of the polytope shown on Figure~\ref{fig:cell3dim}. Moreover, only the colored one-dimensional and two-dimensional faces correspond to zero, while exactly one nonzero monomial corresponds to each of the remaining cells. The proof is by direct computation.
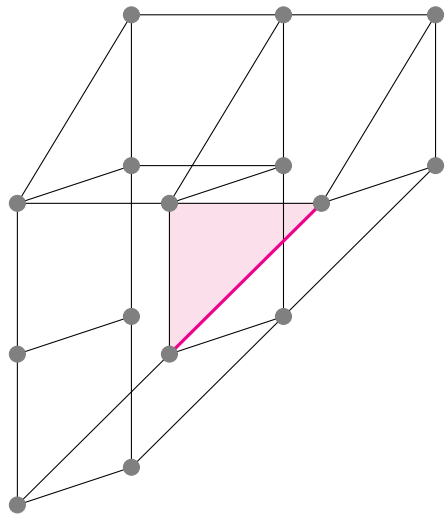
\begin{figure}[h!]

\begin{tikzpicture} 
\filldraw[magenta!15] (2,2) -- (2, 4) -- (4,4);

\draw (0,0) -- (0,4) -- (4, 4) -- (0,0) -- (1.5,0.5) --  (5.5,4.5) -- (5.5,6.5) -- (4, 4);
\draw (3.5,6.5) -- (3.5,2.5) -- (2,2) -- (2,4) -- (3.5,6.5) -- (5.5,6.5);
\draw (1.5,6.5) -- (1.5,0.5);
\draw (2, 4) -- (3.5, 4.5) -- (1.5, 4.5) -- (0, 4) -- (1.5, 6.5) -- (3.5, 6.5);
\draw (4, 4) -- (5.5, 4.5);
\draw (0, 2) -- (1.5, 2.5);

\draw [magenta, very thick] (2,2) -- (4, 4);

\filldraw [gray] (0,0) circle (3pt);
\filldraw [gray] (0,2) circle (3pt);
\filldraw [gray] (0,4) circle (3pt);
\filldraw [gray] (2,2) circle (3pt);
\filldraw [gray] (2,4) circle (3pt);
\filldraw [gray] (4,4) circle (3pt);
\filldraw [gray] (1.5,0.5) circle (3pt);
\filldraw [gray] (3.5,2.5) circle (3pt);
\filldraw [gray] (5.5,4.5) circle (3pt);
\filldraw [gray] (5.5,6.5) circle (3pt);
\filldraw [gray] (3.5,6.5) circle (3pt);
\filldraw [gray] (1.5,6.5) circle (3pt);
\filldraw [gray] (1.5,4.5) circle (3pt);
\filldraw [gray] (1.5,2.5) circle (3pt);
\filldraw [gray] (3.5,4.5) circle (3pt);

\end{tikzpicture}

    \caption{Cellular decomposition of $GZ(3,2,0)$}
    \label{fig:cell3dim}
\end{figure}

\subsection{General case}\label{ssec:ex-general}
Note that if a polynomial $f$ is symmetric over $x_i$ and $x_{i+1}$, then for every polynomial $g$ we have $\pi_i^{(\beta)}(f \cdot g) = f \cdot \pi_i^{(\beta)}(g)$. This means that for every partition $(a,b,c)$ we have the following equality of Lascoux polynomials:
 \[
 \cL_{(a,b,c)}(x_1,x_2,x_3)=(x_1x_2x_3)^c\cdot\cL_{(a-c,b-c,0)}(x_1,x_2,x_3).
 \]
Moreover, Remark~\ref{rem:shifts} states that $GZ(a,b,c)$ is obtained from $GZ(a-c, b-c, 0)$ by a parallel translation by integer vector $(c,c,c)$. 

This means that it is enough to describe the Lascoux polynomials for permutation $u \in S_3$ and partition $\lambda = (a, b, 0)$. 
As before, first put the monomial $x_1^ax_2^b$ in the vertex corresponding to the identity permutation. Then put monomials of the polynomial $\pi_1^{(\beta)}(x_1^ax_2^b)$ along the corresponding one-dimensional face. Under the action of $\pi_2^{(\beta)}$, each monomial of $\mathscr{L}_{s_1, \lambda}$ ``expands'' into a row in a subdivision of a two-dimensional face. For instance, the highest row corresponds to $\mathscr{L}_{s_2, \lambda}$ (see Figure~\ref{fig:example-general}).

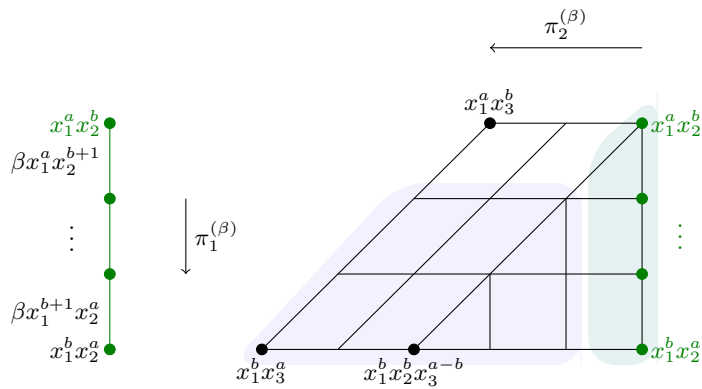
\begin{figure}[h!]
\begin{tikzpicture} [rotate = 180]
\footnotesize
    \draw [PineGreen] (12,0) node [left] {$x_1^ax_2^b$} -- node [black] [left] {$\beta x_1^ax_2^{b+1}$} (12,1) -- (12,2) --node [black] [left] {$\beta x_1^{b+1}x_2^a$}  (12,3)node [black] [left] {$x_1^bx_2^a$} ;
    \draw [->] (11,1) -- node [right] {$\pi_1^{(\beta)}$}(11,2);
    \draw (12.5, 1.5) node [rotate = 90] {\ldots};

    \filldraw [PineGreen] (12,0) circle (2pt);
    \filldraw [PineGreen] (12,1) circle (2pt);
    \filldraw [PineGreen] (12,2) circle (2pt);
    \filldraw [PineGreen] (12,3) circle (2pt);

    \filldraw [blue!5] [rounded corners=10pt](5.8, 0.8) -- (5.8, 3.2) -- (10.4, 3.2) -- (8.2, 0.8) -- (5.8, 0.8) -- (5.8, 3.2);
    \filldraw[teal!10!white] [rounded corners=10pt](4.8, -0.4) -- (4.8, 3.2) -- (5.7, 3.2) -- (5.7, 0.4) -- (4.8, -0.4) -- (4.8, 3.2);

    \draw (5, 0) -- (5,3) -- (10,3) -- (7,0) -- (5,0);
    \draw (5,0) -- (8,3);
    \draw (6,0) -- (9,3);
    \draw (5,1) -- (8,1);
    \draw (5,2) -- (9,2);
    \draw (6,1) -- (6,3);
    \draw (7,2) -- (7,3);

    \filldraw [PineGreen] (5,0) node [right] {$x_1^a x_2^b$} circle (2pt);
    \filldraw [PineGreen] (5,1) circle (2pt);
    \filldraw [PineGreen] (5,2) circle (2pt);
    \filldraw [PineGreen] (5,3) node [right]  {$x_1^b x_2^a$} circle (2pt);
    \draw (4.5, 1.5) node [PineGreen][rotate = 90] {$\ldots$};

    \filldraw  (8,3) node [below] {$x_1^b x_2^b x_3^{a-b}$} circle (2pt);
    \filldraw  (10,3) node [below] {$x_1^b x_3^a$} circle (2pt);
    \filldraw  (7,0) node [above] {$x_1^a x_3^b$} circle (2pt);
    \draw [->] (5,-1) -- node [above] {$\pi_2^{(\beta)}$}(7,-1);
\end{tikzpicture}
\caption{Subdivision of faces for $GZ(a,b,0)$}\label{fig:example-general}
\end{figure}

Next we split the monomials of $\mathscr{L}_{s_2,\lambda}$ into several groups. The first one (colored blue in Figure~\ref{fig:example-general}) can be represented as $ x_3\pi_2^{(\beta)}\pi_1^{(\beta)}(x_1^{a-1}x_2^b)$. Since $x_3$ is symmetric with respect to $x_1$ and $x_2$, it follows that $\pi_1^{(\beta)}(x_3\pi_2^{(\beta)}\pi_1^{(\beta)}(x_1^{a-1}x_2^b)) = x_3 \pi_1^{(\beta)}(\pi_2^{(\beta)}\pi_1^{(\beta)}(x_1^{a-1}x_2^b))$, hence it is the case for smaller $a$ and $b$.
The monomials in the second group (colored green) form a polynomial that is symmetric over $x_1$ and $x_2$; denote it by $f$.  It is easily shown that $\pi_1^{(\beta)}(f) = f$. To complete the construction in is enough to check that the remaining monomials expand into columns of the ``correct'' size, consistent with the size of the polytope $GZ(a, b, 0)$. The proof is by direct computation.

\begin{remark}
    Note that for general $a$ and $b$, ``most'' cells look like (open) standard unit cubes; in particular, this is true for the ``interior'' cells, i.e. those with the closure not meeting the boundary of $GZ(a,b,0)$. Moreover, the cells corresponding to the zero monomial also lie in the boundary of the polytope; these are two cells of dimension 1 or 2 adjacent to the nonsimple vertex, shown in purple in Figure~\ref{fig:cell3dim}. This is the case in general: such cells always belong to the boundary of the polytope and cannot have  dimension $0$ or $n(n-1)/2$. However, for Gelfand--Zetlin polytopes of nondominant weights $\dim GZ(\lambda)<n(n-1)/2$, and cells corresponding to zero monomials can have maximal dimension, i.e. dimension equal to the dimension of the polytope itself.
\end{remark}

\section{Enhanced Gelfand--Zetlin patterns, cellular decomposition and formula for Lascoux polynomials}\label{sec:main}

In this section we give the main results of this paper. We start with constructing a cellular decomposition for $GZ(\lambda)$ and assigning a monomial to each cell. A cell is called \emph{efficient} if this monomial is nonzero and \emph{inefficient} otherwise. Then we state the main theorem: a Lascoux polynomial $\cL_{w,\lambda}$ is equal to the sum of monomials for the cells located in the set of dual Kogan faces corresponding to $w$. The proof of this result is given in the two subsequent sections. 

\subsection{Enhanced Gelfand--Zetlin patterns}\label{ssec:constr} Here we present a construction for cells that provide a cellular decomposition of $GZ(\lambda)$. These cells are convex polytopes without boundary; by construction, the set of zero-dimensional cells will coincide with the set of integer points of $GZ(\lambda)$. This decomposition is regular in the following sense: if the closure of a given cell does not meet the boundary of $GZ(\lambda)$, this cell is just a face of the standard unit cube.

The cells are indexed by the so-called \emph{enhanced Gelfand--Zetlin patterns}, i.e. Gelfand--Zetlin patterns with some additional data, which we call \emph{enhancement}. These data are of two kinds: first, some elements in a pattern may be encircled, and second, some pairs of neighbor elements in consecutive rows can be joined by an edge. 

Informally, a pattern without enhancement stands for the ``maximal'' point of the closure of the corresponding cell, i.e. the point with the largest sum of coordinates. 

\begin{definition}\label{def:enhanced} A Gelfand--Zetlin pattern with the top row $(\lambda_n,\dots,\lambda_1)$ with some entries encircled and with edges between certain neighboring entries is said to be an \emph{enhanced Gelfand--Zetlin pattern}, if these elements satisfy the following conditions:

\begin{enumerate}
	\item\label{cond:1} The numbers in the first row are encircled.
	
	\item\label{cond:2} The two entries joined by an edge must be equal, and the bottom entry should be encircled. The converse does not have to be true: two equal neighboring entries are not necessarily joined by an edge.

	\item\label{cond:3} Two neighboring entries in a row are joined by edges with an entry above them if and only if they are joined with an entry below them. Pictorially this can be presented as follows:
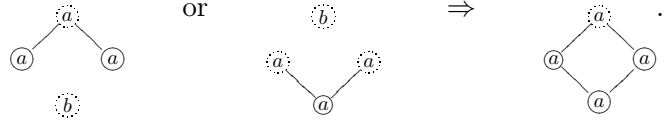
\[
\xymatrixrowsep{1pc}\xymatrixcolsep{1pc}
\scalebox{0.7}{\xymatrix{ & *+[o][F.]{a} \ar@{-}[dl]\ar@{-}[dr]\\ *+[Fo]{a}&& *+[Fo]{a}\\ & *+[o][F.]{b}}}\qquad\text{or}\qquad
\scalebox{0.7}{\xymatrix{ & *+[o][F.]{b} \\  *+[o][F.]{a}&&  *+[o][F.]{a}\\ & *+[o][F]{a}\ar@{-}[ul]\ar@{-}[ur]}}\qquad\Rightarrow\qquad
\scalebox{0.7}{\xymatrix{ &  *+[o][F.]{a} \ar@{-}[dl]\ar@{-}[dr]\\  *+[o][F]{a}&&  *+[Fo]{a}\\ & *+[Fo]{a}\ar@{-}[ul]\ar@{-}[ur]}}.
\]
(a dotted circle around an entry means that it may be either encircled or not).

\item\label{cond:4} If two entries in the first row are equal, then the entry below them (which is equal to both of them) is encircled and connected to both of them by edges.

\item\label{cond:5} If $a<b$ and the pattern contains the following triangle: $\begin{smallmatrix} a&&b \\ &a
\end{smallmatrix}$, then the lowest entry in the triangle is encircled.  Pictorially:
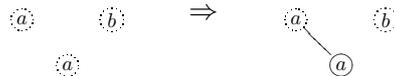
\[
\xymatrixrowsep{1pc}\xymatrixcolsep{1pc}
\scalebox{0.7}{\xymatrix{ *+[o][F.]{a} &&  *+[o][F.]{b}\\& *+[o][F.]{a}}}\qquad\Rightarrow\qquad \scalebox{0.7}{\xymatrix{ *+[o][F.]{a} &&  *+[o][F.]{b}\\& *+[Fo]{a}\ar@{-}[ul]}}
\] 

\item\label{cond:6} If $a<b$ and the pattern contains the following triangle: $\begin{smallmatrix}a && b\\ &b
\end{smallmatrix}$ with the bottom entry encircled, then there is an edge between the two $b$'s:
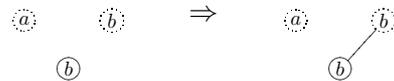
\[
\xymatrixrowsep{1pc}\xymatrixcolsep{1pc}
\scalebox{0.7}{\xymatrix{ *+[o][F.]{a} &&  *+[o][F.]{b}\\& *+[o][F]{b}}}\qquad\Rightarrow\qquad \scalebox{0.7}{\xymatrix{ *+[o][F.]{a} &&  *+[o][F.]{b}\\& *+[o][F]{b}\ar@{-}[ur]}}
\]

\item\label{cond:7} For a triangle $\begin{smallmatrix}
	a && a\\ &a
\end{smallmatrix}$: if the two top entries can be connected by a path of edges, the bottom entry should be encircled and connected with them.

\item\label{lastcondition} If in a triangle $\begin{smallmatrix}
	a && a\\ &a
\end{smallmatrix}$ the bottom entry is encircled, then it should be connected with at least one of them by an edge.

\end{enumerate}
We denote the set of all enhanced patterns with the first row $\lambda$ by $\cP(\lambda)$.
\end{definition}

\begin{example}\label{ex:dualkogan} The pattern
$\begin{smallmatrix} 0 && 1&& 2\\ &1&& 2\\ &&2\end{smallmatrix}$ has eight enhancements.
\[
\xymatrixrowsep{1pc}\xymatrixcolsep{1pc}
\scalebox{0.7}
{\xymatrix{
*+[o][F]{0} && *+[o][F]{1} && *+[o][F]{2}\\
& 1 &&2 \\
&&2}
}\qquad\qquad
\scalebox{0.7}
{\xymatrix{
*+[o][F]{0} && *+[o][F]{1} && *+[o][F]{2}\\
& *+[o][F]{1}\ar@{-}[ur]  &&2 \\
&&2}
}\qquad\qquad
\scalebox{0.7}
{\xymatrix{
*+[o][F]{0} && *+[o][F]{1} && *+[o][F]{2}\\
& {1}  &&*+[o][F]{2}\ar@{-}[ur] \\
&&2}
}\qquad\qquad
\scalebox{0.7}
{\xymatrix{
*+[o][F]{0} && *+[o][F]{1} && *+[o][F]{2}\\
& {1}  && 2 \\
&&*+[o][F]{2}\ar@{-}[ur]}
}\]
\[
\xymatrixrowsep{1pc}\xymatrixcolsep{1pc}
\scalebox{0.7}{\xymatrix{
*+[o][F]{0} && *+[o][F]{1} && *+[o][F]{2}\\
& *+[o][F]{1}\ar@{-}[ur]  && *+[o][F]{2}\ar@{-}[ur] \\
&&2}
}\qquad\qquad
\scalebox{0.7}{\xymatrix{
*+[o][F]{0} && *+[o][F]{1} && *+[o][F]{2}\\
&1 && *+[o][F]{2}\ar@{-}[ur] \\
&&*+[o][F]{2}\ar@{-}[ur]}
}\qquad\qquad
\scalebox{0.7}{\xymatrix{
*+[o][F]{0} && *+[o][F]{1} && *+[o][F]{2}\\
& *+[o][F]{1}\ar@{-}[ur]   && 2 \\
&&*+[o][F]{2}\ar@{-}[ur]}
}\qquad\qquad
\scalebox{0.7}{\xymatrix{
*+[o][F]{0} && *+[o][F]{1} && *+[o][F]{2}\\
& *+[o][F]{1}\ar@{-}[ur]  && *+[o][F]{2}\ar@{-}[ur] \\
&&*+[o][F]{2}\ar@{-}[ur]}
}
\]
\end{example}

\begin{example}\label{ex:nonsimple} The pattern $\begin{smallmatrix}0 &&1 && 2\\ &1&& 1\\ &&1\end{smallmatrix}$ has four enhancements. 
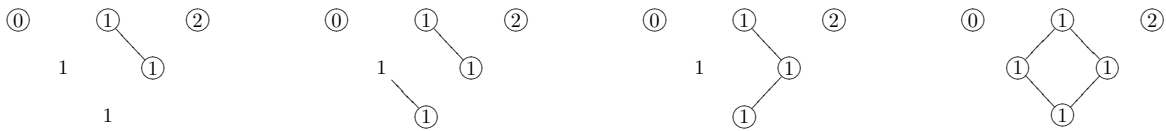
\[
\xymatrixrowsep{1pc}\xymatrixcolsep{1pc}
\scalebox{0.7}{\xymatrix{
*+[o][F]{0} && *+[o][F]{1} && *+[o][F]{2}\\
& 1  && *+[o][F]{1}\ar@{-}[ul] \\
&&1}
}
\qquad\qquad
\scalebox{0.7}{\xymatrix{
*+[o][F]{0} && *+[o][F]{1} && *+[o][F]{2}\\
& 1  && *+[o][F]{1}\ar@{-}[ul] \\
&&*+[o][F]{1}\ar@{-}[ul]}
}
\qquad\qquad
\scalebox{0.7}{\xymatrix{
*+[o][F]{0} && *+[o][F]{1} && *+[o][F]{2}\\
& 1  && *+[o][F]{1}\ar@{-}[ul] \\
&&*+[o][F]{1}\ar@{-}[ur]}
}
\qquad\qquad
\scalebox{0.7}{\xymatrix{
*+[o][F]{0} && *+[o][F]{1} && *+[o][F]{2}\\
& *+[o][F]{1}\ar@{-}[ur]  && *+[o][F]{1}\ar@{-}[ul] \\
&&*+[o][F]{1}\ar@{-}[ur]\ar@{-}[ul]}
}
\]
Note that according to Definition~\ref{def:enhanced}~(\ref{cond:5}), the last entry in the second row must be encircled and connected to the middle entry in the first row.
\end{example}

An enhanced pattern can be viewed as a graph (with marked vertices). Consider the connected components of this graph. 

\begin{lemma} The connected components of an enhanced Gelfand--Zetlin pattern satisfy the following:
\begin{enumerate}
	\item\label{cc:cond1} the entries in the first row belong to the same connected component if and only if they have the same value;
	\item\label{cc:cond2} each connected component either has a unique highest vertex or contains at least one entry from the first row;
	\item\label{cc:cond3} all vertices in a connected component, possibly except the highest one, are encircled. In particular, the number of connected components is not less than the number of distinct $\lambda_i$'s plus the number of entries without circles.
\end{enumerate}
\end{lemma}

\begin{proof} This follows immediately from Definition~\ref{def:enhanced}. We only comment on~(\ref{cc:cond2}), which may be not completely obvious. Suppose there exists a connected component of an enhanced Gelfand--Zetlin pattern having two distinct highest vertices $a_{ij}$ and $a_{ik}$ in the $i$-th row, for $i\geq 1$ and $j<k$. There is a path connecting these entries; applying Definition~\ref{def:enhanced} (\ref{cond:3}) repeatedly, we see that every two neighbor elements in the ``interior'' of this path are linked by an edge (in 
 particular, they are all equal). This means that, in particular, $a_{ij}$ and $a_{i,j+1}$ are connected by edges with $a_{i+1,j}$, which is encircled. So, condition~(\ref{cond:3}) implies that $a_{i-1,j+1}$ is connected with $a_{ij}$, as well as with $a_{i,j+1}$, and hence the latter elements are not the highest ones in their connected component.
 
 Figure~\ref{fig:unique} shows how existence of a path between two elements in the same row, marked by squares in the left picture,  implies additional equalities, shown by dashes in the right picture.
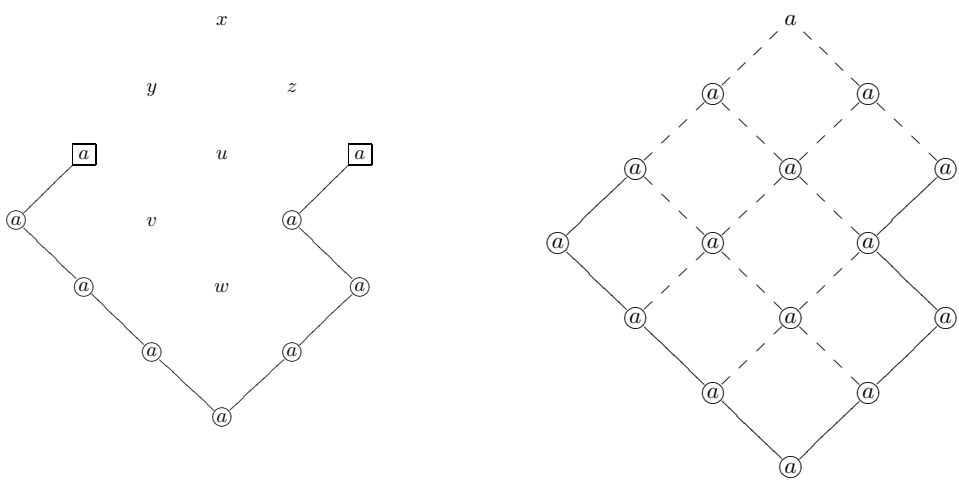
\begin{figure}[h!]
\[
\scalebox{0.7}{\xymatrix{
&&& x\\
&& y&& z\\
&*+[oo][F]{a} && {u} && *+[oo][F]{a}&\\
*+[o][F]{a}\ar@{-}[ur]  && {v} && *+[o][F]{a}\ar@{-}[ur] \\
&*+[o][F]{a}\ar@{-}[ul] && {w} && *+[o][F]{a}\ar@{-}[ul]\\
&& *+[o][F]{a}\ar@{-}[ul] && *+[o][F]{a}\ar@{-}[ur]\\
&&& *+[o][F]{a}\ar@{-}[ul]\ar@{-}[ur]\\
}
}
\qquad\qquad
\scalebox{0.8}{\xymatrix{
&&&{a}\\
&&*+[o][F]{a}\ar@{--}[ur] && *+[o][F]{a}\ar@{--}[ul] \\
&*+[o][F]{a}\ar@{--}[ur] && *+[o][F]{a}\ar@{--}[ul] \ar@{--}[ur] && *+[o][F]{a}\ar@{--}[ul] &\\
*+[o][F]{a}\ar@{-}[ur]  && *+[o][F]{a}\ar@{--}[ur]\ar@{--}[ul] && *+[o][F]{a}\ar@{-}[ur]\ar@{--}[ul] \\
&*+[o][F]{a}\ar@{--}[ur]\ar@{-}[ul] && *+[o][F]{a}\ar@{--}[ur]\ar@{--}[ul] && *+[o][F]{a}\ar@{-}[ul]\\
&& *+[o][F]{a}\ar@{-}[ul]\ar@{--}[ur] && *+[o][F]{a}\ar@{-}[ur]\ar@{--}[ul]\\
&&& *+[o][F]{a}\ar@{-}[ul]\ar@{-}[ur]\\
}
}
\]
\caption{Uniqueness of highest vertex in a connected component}\label{fig:unique}	
\end{figure}

\end{proof}

The statement~(\ref{cc:cond3}) from the previous lemma motivates the following definitions.
\begin{definition}\label{def:rankpattern}
	The \emph{rank} $\rk P$ of an enhanced pattern $P$ is the number of entries without circles.
\end{definition}

\begin{definition}\label{def:ineffpattern}
	An enhanced pattern $P$ is said to be \emph{inefficient} if it contains a triangle of the form $\begin{smallmatrix}a && a\\ &a
\end{smallmatrix}$ such that its bottom entry is not connected with the right one by an edge, and \emph{efficient} otherwise. The set of all efficient enhanced patterns with the first row $\lambda$ is denoted by $\cP^+(\lambda)$.
\end{definition}

For instance, in Example~\ref{ex:nonsimple} the first two enhanced patterns are inefficient, while the last two are efficient, and all patterns in Example~\ref{ex:dualkogan} are efficient.

\begin{prop}\label{prop:zeroefficient} Every enhanced pattern of rank zero is efficient.
\end{prop}

\begin{proof}
Take an inefficient enhanced pattern $P$. This means that it contains a triangle of the form $\begin{smallmatrix}a && a\\ &a\end{smallmatrix}$ such that there is no edge between the bottom and the right entries. Definition~\ref{def:enhanced}~(\ref{cond:7}) implies that these two entries are contained in different connected components, both marked with the same number $a$. This means that at least one of these components contains a vertex without circle, so the rank of $P$ cannot be zero.
\end{proof}

Moreover, it turns out that for an efficient enhanced pattern, the edges provide redundant data. Namely, we have the following lemma.

\begin{lemma}\label{lem:redundant}
	The edges in an efficient enhanced pattern are uniquely determined by positions of encircled vertices.
\end{lemma}

%\todo[inline]{Write in the form of an algorithm}

\begin{proof} Given a set of encircled vertices in an efficient enhanced pattern, we can uniquely reconstruct the set of edges as follows. Let us scan the enhanced pattern row by row, from top to bottom, starting from the second row, and draw edges going up from some of the encircled vertices. For each encircled entry $a_{ij}$, consider the two elements $a_{i-1,j}$ and $a_{i-1,j+1}$ above it. If both of them are different from $a_{ij}$, then there is no edge going up from $a_{ij}$. If only one of them is equal to $a_{ij}$, and the other is not, we must join $a_{ij}$ with the element equal to it, according to Definition~\ref{def:enhanced},~(\ref{cond:5}) and~(\ref{cond:6}).

Finally, if $a_{ij}$ is equal to the both entries  $a_{i-1,j}$ and $a_{i-1,j+1}$, two cases may occur. If there is a path of edges situated in rows with numbers less than or equal to $i-1$ and joining these two entries, then $a_{ij}$ must be joined with both of them, as prescribed by Definition~\ref{def:enhanced},~(\ref{cond:7}). Otherwise, according to condition~(\ref{lastcondition}) it should be connected with only one of these two entries; since our enhanced pattern is efficient, this is the right entry $a_{i-1,j+1}$. This procedure produces the set of edges in a unique way.
\end{proof}

%\begin{proof} The conditions listed in Definition~\ref{def:enhanced} imply that positions of edges are defined by positions of encircled vertices in all cases except for case~(\ref{lastcondition}). In the latter case there are two possibilities of joining the bottom vertex in the triangle $\begin{smallmatrix}a&& a\\&a\end{smallmatrix}$ with one of its neighbors in the upper row, and only one of them defines an efficient pattern.
%\end{proof}

For an efficient enhanced GZ-pattern $P$, we assign to it a monomial $x^P$ in the following way. Let $S_i(P)$ be the sum of numbers in the $i$-th row of the pattern $P$, with $S_0(P)=\lambda_1+\dots+\lambda_n$, and let $D_i(P)$ stand for the number of entries without circles in the $i$-th row of $P$. Denote $d_{n+1-i}=d_{n+1-i}(P)=S_{i-1}(P)-S_i(P)+D_i(P)$. Then
\[
x^P=\beta^{\rk P}x_1^{d_1}\dots x_n^{d_n}. 
\]
For an inefficient enhanced GZ-pattern $P$ we formally set $x^P=0$.

In the next section we construct a cellular decomposition of $GZ(\lambda)$, with cells corresponding to enhanced patterns, and with the dimension of a cell being equal to the rank of its enhanced pattern. Some of these cells will correspond to monomials in Lascoux polynomials; as we will see, these will be exactly the cells constructed from the efficient enhanced patterns. This is the motivation behind Definition~\ref{def:ineffpattern}.

\subsection{Cellular decomposition of Gelfand--Zetlin polytope}\label{ssec:monomials}
In this subsection, we construct a cellular decomposition of $GZ(\lambda)$. The cells of this decomposition are indexed by enhanced patterns; moreover, the dimension of a cell is equal to the rank of the corresponding enhanced pattern.

Consider an enhanced Gelfand--Zetlin pattern $P$. For each such pattern we write a set of equalities and inequalities that, together with the inequalities defining $GZ(\lambda)$, defines a subset $C_P\subset GZ(\lambda)$. As we will show further, these sets are pairwise disjoint, open in their affine spans and homeomorphic to open balls; they define a cellular decomposition of $GZ(\lambda)$ compatible with the polytope structure (i.e., every face of $GZ(\lambda)$ is also a union of cells).

Recall that we denote the coordinates in $\RR^{\frac{n(n-1)}2}\supset GZ(\lambda)$ by $y_{ij}$, with $1\leq i\leq n-1$ and $1\leq j\leq n+1-i$. We also fix the topmost row of a Gelfand--Zetlin tableau by setting $y_{0,j}=\lambda_{n+1-j}$. The inequalities that define the polytope are given by the tableau~(\ref{eq:gz-tableau}) on p.~\pageref{eq:gz-tableau}: these are
\[
y_{i-1,j}\leq y_{ij}\leq y_{i-1,j+1},
\]
for each $(i,j)$ in the aforementioned range.

Now we define the cellular decomposition of $GZ(\lambda)$.

\begin{construction}\label{def:cells} Let $P$ be an enhanced pattern with entries $a_{ij}$. To each coordinate $y_{ij}$ we assign an equality or a double inequality as follows:
\begin{enumerate}
	\item if there is an edge going up from $a_{ij}$ to $a_{i-1,j}$ (resp. to $a_{i-1,j+1}$), then $y_{ij}=y_{i-1,j}$ (resp. $y_{ij}=y_{i-1,j+1}$);
	\item if there are no edges going up from $a_{ij}$, and this entry is encircled, then $y_{ij}=a_{ij}$;
	\item if there are no edges going up from $a_{ij}$ and this entry is not encircled, we impose a double inequality on $y_{ij}$ as follows:
 \begin{enumerate}
     \item If the entry $a_{i-1,j}$ satisfies $a_{ij}-a_{i-1,j}\geq 2$, then $a_{ij}-1<y_{ij}$; otherwise, $y_{i-1,j}<y_{ij}$;
     \item If $a_{i-1,j+1}$ is equal to $a_{ij}$, we set $y_{ij}<y_{i-1,j+1}$; otherwise, $y_{ij}<a_{ij}$.
 \end{enumerate}
\end{enumerate}
Denote the set  defined by these equalities and inequalities by $\widehat{C_P}$. This is ``almost'' the required cell corresponding to $P$; however, it does not necessarily lie in $GZ(\lambda)$. To get an actual cell, take the affine span $L$ of $\widehat{C_P}$ and intersect $\widehat{C_P}$ with the relative interior of $GZ(\lambda)\cap L$ in $L$:
\[
C_P=\widehat{C_P}\cap (GZ(\lambda)\cap L)^0.
\]
This set is convex and open in $L$ (as the intersection of two open  convex sets); we shall see in Lemma~\ref{lem:52} that it is nonempty.
\end{construction}

This means the following. For each connected component in $P$ containing only encircled entries with the same numbers, all the corresponding coordinates are equal to this number. On the other hand, if a connected component has a non-encircled vertex, the corresponding coordinate can take values in an interval determined by the condition (\ref{cond:4}) of~Definition~\ref{def:enhanced}; note that the length of this interval does not exceed $i-1$, where $i$ is the row number. All the remaining coordinates in the same connected component (corresponding to encircled entries) are equal to this coordinate.

\begin{example}\label{ex:patterncell} Consider the following pattern $P$ and construct the set of inequalities on $y_{ij}$ defining the  cell $C_P$ corresponding to it. The right diagram provides the set of inequalities defining the Gelfand--Zetlin polytope in $\RR^6$.
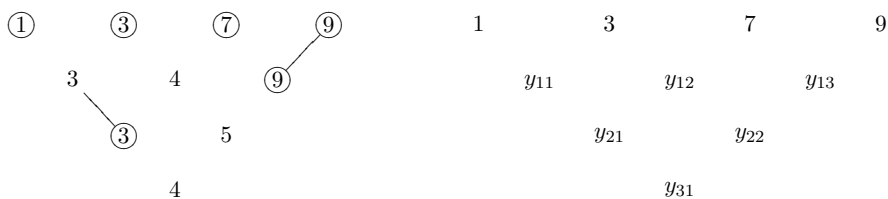
\[
\xymatrixrowsep{1pc}\xymatrixcolsep{1pc}
\scalebox{0.8}{
\xymatrix{
*+[o][F]{1} && *+[o][F]{3} && *+[o][F]{7} && *+[o][F]{9}\\
& 3 && 4 && *+[o][F]{9}\ar@{-}[ur] \\
&& *+[o][F]{3}\ar@{-}[ul] && 5\\
&&&4}
}\qquad\qquad
\scalebox{0.8}{
\xymatrix{
{1} && {3} && {7} && {9}\\
& y_{11} && y_{12} && y_{13} \\
&& y_{21} && y_{22}\\
&&&y_{31}}
}
\]
The first row defines the following double inequalities and equality:
\[
2<y_{11}<3,\quad 3<y_{12}<4,\quad y_{13}=9.
\]
Now consider the second row. An edge from $a_{21}$ going up prescribes us to impose the equality $y_{21}=y_{11}$. From $a_{22}=5$ situated below $4$ and an encircled $9$, we get inequality $y_{12}<y_{22}<5$. Likewise, the third row defines double inequality $y_{21}<y_{31}<4$.

So the set $\widehat{C_P}$ is defined by the following inequalities and equalities:
\[
2< y_{11}<3,\quad 3<y_{12}<4,\quad y_{13}=9,\quad y_{21}=y_{11},\quad y_{12}<y_{22}<5,\quad y_{21}<y_{31}<4.
\]
Note that $\widehat{C_P}$ is not a subset of $GZ(9,7,3,1)$. Indeed, the integer point given by the following tableau, which is not a Gelfand--Zetlin pattern:
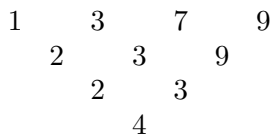
\[
\begin{matrix}
	1 && 3&& 7 &&9\\
	&2 && 3&& 9\\
	&&2&& 3\\
	&&&4
\end{matrix}
\]
belongs to the closure of $\widehat{C_P}$, but not to $GZ(9,7,3,1)$. To get a cell in the Gelfand--Zetlin polytope, we need to take the affine span of $\widehat{C_P}$, which is a subspace $L$ defined by the equations $y_{13}=9$, $y_{21}=y_{11}$, intersect the interior of $GZ(9,7,3,1)$ with $L$, obtaining the following subset:
\[
1< y_{11}<3,\quad 3<y_{12}<7,\quad y_{13}=9,\quad y_{21}=y_{11},\quad y_{12}<y_{22}<9,\quad y_{21}<y_{31}<y_{22}.
\]
The intersection of this subset with $\widehat{C_P}$ is the desired cell $C_P$. It is defined by the following inequalities and equalities:
\[
2< y_{11}<3,\quad 3<y_{12}<4,\quad y_{13}=9,\quad y_{21}=y_{11},\quad y_{12}<y_{22}<5,\quad y_{21}<y_{31}<\min(4,y_{22}).
\]
\end{example}

\subsection{Main results}

The following theorems are the main results of this paper.

\begin{theorem}\label{thm:cellular} The cells $C_P$ for $P\in\cP(\lambda)$ form a cellular decomposition of $GZ(\lambda)$.
\end{theorem}

The proof of this theorem is given in \S~\ref{ssec:51}.

\begin{theorem}\label{thm:main1}
Let $\lambda$ be a partition. Then the Grassmannian Grothendieck polynomial $G_{\lambda}^{(\beta)}$ can be computed as follows:
$$
G_{\lambda}^{(\beta)} = \mathscr{L}_{w_0, \lambda} = \sum\limits_{P \in \cP^+(\lambda)}{x^P}.
$$
\end{theorem}

The proof of this theorem is given in \S\;\ref{ssec:52}.

We can generalize this result to get a description of Lascoux polynomials. Denote by $\cP^+(w, \lambda)$ the set of efficient patterns such that the corresponding cells are contained in the union of dual Kogan faces corresponding to $GZ(\lambda)$ and permutation $w$.

\begin{theorem}\label{thm:main2}
Let $w \in S_n$ be a permutation and $\lambda$ be a partition. Then the Lascoux polynomial $\cL_{w,\lambda}$ is equal to
$$
\mathscr{L}_{w, \lambda} = \sum\limits_{P \in \cP^+(w,\lambda)}{x^P}.
$$
\end{theorem} 

We give the proof of this theorem in Section~\ref{sec:proofgen}.

\begin{corollary}\label{cor:positive} Let $\lambda$ be a partition and $u,w \in S_n$  be permutations such that $u \leq w$ in the Bruhat order on $S_n$. Then the polynomial $\mathscr{L}_{w, \lambda}-\mathscr{L}_{u,\lambda}$ has nonnegative coefficients.
\end{corollary}

\begin{proof} Denote the union of dual Kogan faces of $GZ(\lambda)$ corresponding to permutation $w$ by $\Gamma_{w, \lambda}$. The definition of Kogan faces in terms of subwords (cf.~\S\,\ref{ssec:dualkogan})  implies that $\Gamma_{u, \lambda} \subseteq \Gamma_{w, \lambda}$ if $u \leq w$ in the Bruhat order. Applying Theorem~\ref{thm:main2} completes the proof.
\end{proof}

\section {Proofs of the main results for the case of the longest permutation}\label{sec:proofs-w0}

\subsection {Proof of Theorem~\ref{thm:cellular}}\label{ssec:51}
In this section we show that the cells described in \S\,\ref{ssec:monomials} form a cellular decomposition of $GZ(\lambda)$. We split the proof into several lemmas.

\begin{lemma}\label{lem:51} For every point $y \in GZ(\lambda)$ there exists a unique cell $C_P$ such that $y \in C_P$. %Every cell $C_P$ is nonempty.
\end{lemma}

\begin{proof} We describe an algorithm of constructing such a $P$ for any point $y\in GZ(\lambda)$. First take the table of coordinates of $y$ and draw the edges for all pairs of equal neighboring coordinates, no matter whether they are integer or not. Then we need to assign integer values $a_{ij}$ to all the vertices of $P$. This is done for the top vertex in each connected component, starting from the first row. Here we follow Construction~\ref{def:cells}.

We construct the enhanced pattern $P$ row by row, going from top to bottom. The first  row of the pattern is given by $\lambda$. Now suppose we have already filled the row number $i-1$; consider the $i$-th row. If an element $y_{ij}$ is equal to $y_{i-1,j}$ or $y_{i-1,j+1}$ (or both), this means that there is an edge going up from this position; then we let $a_{ij}$ be $a_{i-1,j}$ or $a_{i-1,j+1}$ correspondingly. We encircle this entry.

Otherwise, we distinguish between three cases. If $a_{i-1,j}+1\leq y_{ij}< a_{i-1,j+1}$ and $y_{ij}$ is integer, we set $a_{ij}=y_{ij}$ and encircle this element. If $y_{ij}<a_{i-1,j}+1$, we set $a_{ij}=a_{i-1,j}+1$ and do not encircle it. Finally, for $y_{ij}\notin \ZZ$ and $a_{i-1,j}+1<y_{ij}$, we set $a_{ij}=\lceil y_{ij}\rceil$, and the corresponding entry also has no circle. Construction~\ref{def:cells} implies that for the enhanced pattern $P$ obtained in such a way, the corresponding cell $C_P$ contains the point $y$.
\end{proof}

\begin{example}\label{ex:pointcell} To see how this algorithm works, consider the following tableau and the corresponding point $y\in GZ(1,3,7,9)\subset \RR^6$:
\[
\xymatrixrowsep{1pc}\xymatrixcolsep{1pc}
\scalebox{0.8}
{\xymatrix{
*+[o][F]{1} && *+[o][F]{3} && *+[o][F]{7} && *+[o][F]{9}\\
& 2.5 &&3.1 && 9 \\
&&2.5 && 3.8\\
&&&3.7}
}
\]
First join every pair of equal numbers by an edge and encircle the lower element in each pair. 
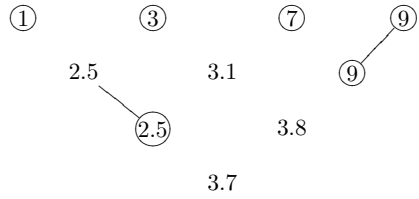
\[
\xymatrixrowsep{1pc}\xymatrixcolsep{1pc}
\scalebox{0.8}
{\xymatrix{
*+[o][F]{1} && *+[o][F]{3} && *+[o][F]{7} && *+[o][F]{9}\\
& 2.5 &&3.1 && *+[o][F]{9}\ar@{-}[ur] \\
&& *+[o][F]{2.5}\ar@{-}[ul] && 3.8\\
&&&3.7}
}
\]
Then replace all the entries in the second row by integers, following the algorithm from the proof of Lemma~\ref{lem:51}. We replace $2.5$ by $3$, and $3.1$ by $4$. Since $9$ is the lower end of an edge, its entry remains equal to the entry on the upper end of it. We obtain the following (note that this intermediate result is not an enhanced Gelfand--Zetlin pattern!).
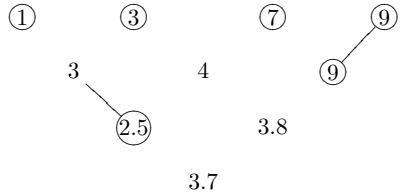
\[
\xymatrixrowsep{1pc}\xymatrixcolsep{1pc}
\scalebox{0.8}
{\xymatrix{
*+[o][F]{1} && *+[o][F]{3} && *+[o][F]{7} && *+[o][F]{9}\\
& 3 && 4 && *+[o][F]{9}\ar@{-}[ur] \\
&& *+[o][F]{2.5}\ar@{-}[ul] && 3.8\\
&&&3.7}
}
\]
Now proceed with the third row. Since there is an edge going up from its first entry, we replace this entry by the integer on top of this edge, i.e. by $3$. With $y_{22}=3.8$, we have $y_{22}<a_{12}+1=5$, so we need to replace $3.8$ by $5$; we do not encircle it.
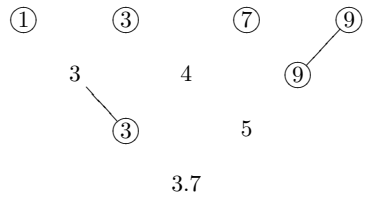
\[
\xymatrixrowsep{1pc}\xymatrixcolsep{1pc}
\scalebox{0.8}
{\xymatrix{
*+[o][F]{1} && *+[o][F]{3} && *+[o][F]{7} && *+[o][F]{9}\\
& 3 && 4 && *+[o][F]{9}\ar@{-}[ur] \\
&& *+[o][F]{3}\ar@{-}[ul] && 5\\
&&&3.7}
}
\] 
 The last stage is modifying $y_{31}=3.7$. We have $y_{31}\leq a_{21}+1=4$, so we set $a_{31}=4$ and obtain the desired answer: the enhanced pattern of the cell containing $y$ looks as follows.
 \[
\xymatrixrowsep{1pc}\xymatrixcolsep{1pc}
\scalebox{0.8}
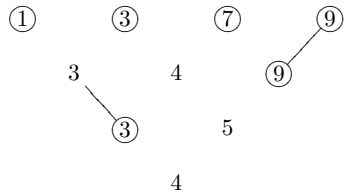
{\xymatrix{
*+[o][F]{1} && *+[o][F]{3} && *+[o][F]{7} && *+[o][F]{9}\\
& 3 && 4 && *+[o][F]{9}\ar@{-}[ur] \\
&& *+[o][F]{3}\ar@{-}[ul] && 5\\
&&&4}
}
\]  
\end{example}

\begin{lemma}\label{lem:52} Every cell $C_P$ is nonempty and  homeomorphic to an open ball of dimension $\rk P$. 
\end{lemma}

\begin{proof} First, the set $\widehat{C_P}$ defined by inequalities in Construction~\ref{def:cells} is a convex set open in its affine span $L$. Replacing these strict inequalities by  non-strict ones, we obtain the closure of $\widehat{C_P}$. It contains the point with coordinates $y_{ij}=a_{ij}$, so this closure is nonempty, and $C_P$ is nonempty as well. Similarly, the relative interior $(L\cap GZ(\lambda))^0$ is convex and nonempty. So it remains to show that their intersection is nonempty.

Indeed, consider the point $y=(y_{ij})$ defined by $y_{ij}=a_{ij}$; it belongs to the closure of both $\widehat C_P$ and $(L\cap GZ(\lambda))^0$. Moreover, in a neighborhood of $y$ both these sets coincide. So $C_P$ also has dimension $\rk P$; being the intersection of two nonempty  convex bounded open sets, it is homeomorphic to an open ball.
\end{proof}

\begin{lemma}\label{lem:54} Let $y\in GZ(\lambda)$, and let $C_P$ be the cell containing it. If for some cell $C$ we have $y\in\overline{C}$, then $C_P\subset \overline{C}$. 
\end{lemma}

\begin{proof} Given $y\in GZ(\lambda)$, let us find all enhanced patterns $Q$ such that $y\in \overline{C_Q}$. This will be done similarly to the proof of Lemma~\ref{lem:51}.

We construct all such patterns row by row from top to bottom. On each step, the procedure may not be unique. The first row is given by $\lambda$; this is the induction base.

Now suppose that the first $i-1$ rows are filled, and consider the coordinates $y_{ij}$ in the $i$-th row, starting from the first one. For a given coordinate $y_{ij}$,  we proceed exactly as in Lemma~\ref{lem:51}. Namely, if $y_{ij}<a_{i-1,j}+1$, we set $a_{ij}=a_{i-1,j}+1$; otherwise, if it is not an integer, we set $a_{ij}=\lceil y_{ij}\rceil$. In both of these cases, the corresponding entry has no circle. 

If $y_{ij}$ is an integer and $a_{i-1,j}+1\leq y_{ij}$, we can construct the corresponding entry of the pattern in at most three different ways. In the first case, we set $a_{ij}=y_{ij}$, put a circle around this vertex and join it with entries in the previous row, just like in Lemma~\ref{lem:51}. In the second and the third case, we set $a_{ij}=y_{ij}$ or $a_{ij}=y_{ij}+1$ and do not put a circle around this vertex, provided that the resulting diagram (or the constructed part of it) satisfies the conditions of Definition~\ref{def:enhanced}.

By construction, for all the patterns $Q$ obtained by this procedure the corresponding cell closure $\overline C_Q$ contains $y$, and this set includes $C_P$. Moreover, for every constructed $Q$ we have $\rk P\leq\rk Q$, with the equality only in the case $P=Q$. It is also clear that for all points $y\in C_P$ the set of such patterns $Q$ will be the same.

% If an element $y_{ij}$ is equal to any of two elements upstairs (or both), then the corresponding $a_{ij}$ is already determined. If $y_{ij}$ is an integer, we set $a_{ij}=y_{ij}$ and encircle the entry $a_{ij}$.  If it is not an integer, then two cases may occur: if $y_{ij}<a_{i-1,j}+1$, we set $a_{ij}=a_{i-1,j}+1$; otherwise, we set $a_{ij}=\lceil y_{ij}\rceil$. In both of these cases, the corresponding entry has no circle. Construction~\ref{def:cells} implies that for the enhanced pattern $P$ obtained in such a way, the corresponding cell $C_P$ contains the point $y$.

%  If $y_{ij}$ is an integer, we set $a_{ij}=y_{ij}$. Otherwise, three cases may occur: if $y_{ij}<a_{i-1,j}+1$, we set $a_{ij}=a_{i+1,j}+1$. Otherwise we set $a_{ij}=\lfloor y_{ij}\rfloor$ or $a_{ij}=\lceil y_{ij}\rceil$.

%This also follows from %Construction~\ref{def:cells}: every point in the boundary of $C_P$ is obtained by replacing some of the double inequalities defining $C_P$ by equalities, and this defines an enhanced pattern $Q$. By construction, we have $C_Q\subset\overline{C_P}$.
%Let $y \in \partial C=\overline C\setminus C$. Note that $C$ is not the cell of smallest dimension containing $y$, and there exists a cell $C_x$ such that $y \in C_y \subseteq \partial C$.
\end{proof}

This lemma immediately implies that the boundary $\overline C\setminus C$ of each cell consists of cells of smaller dimension. So $GZ(\lambda)=\bigsqcup_{P\in\cP(\lambda)} C_P$, is indeed a cellular decomposition. This concludes the proof of Theorem~\ref{thm:cellular}.

\subsection{Proof of Theorem~\ref{thm:main1}}\label{ssec:52}

In this section we establish a bijection between the set $\cP^+(\lambda)$ of efficient enhanced patterns and the multi-set of monomials in $\pi_{w_0}^{(\beta)}(x^\lambda)$.
%
%Recall that the Demazure--Lascoux operator $\pi_{i}^{(\beta)}$ is idempotent, i.e., it satisfies 
%\[
%\pi_{i}^{(\beta)}(f - \pi_{i}^{(\beta)}(f))= 0.
%\]

Denote by $c_k$ the following Coxeter element $s_k\dots s_1\in S_k$; here $1\leq k\leq n-1$. The longest permutation $w_0$ can be presented as the product of such elements:
\begin{equation}\label{eq:w0}
w_0=s_1(s_2s_1)\dots(s_{n-1}\dots s_1)=c_1 c_2\dots c_{n-1}.
\end{equation}

\begin{definition}
	A polynomial $p(\beta,x_1,\dots,x_n)$ is said to be \emph{multiplicity free}, if all its nonzero coefficients are equal to 1.
\end{definition}

\begin{lemma}\label{lem:multfree} Let $\mu$ be a partition. For a monomial $x^\mu= x_1^{\mu_1}\dots x_n^{\mu_n}$, the polynomial $\pi_{c_k}^{(\beta)} x^\mu$ is multiplicity free.	
\end{lemma}

\begin{proof}
	The proof is by induction on $k$. If $k=1$, the statement is obvious:
	\[
	\pi_{c_1}^{(\beta)}x^\mu=\pi_1^{(\beta)}x^\mu=\left( (x_1^{\mu_1}x_2^{\mu_2}+\dots+x_1^{\mu_2}x_2^{\mu_1})+\beta (x_1^{\mu_1}x_2^{\mu_2+1}+\dots+x_1^{\mu_2+1}x_2^{\mu_1})\right)\cdot x_3^{\mu_3}\dots x_n^{\mu_n}.
	\]
	Note that all monomials in the right-hand side have different bidegrees over $(\beta,x_1)$.
	
	Applying $\pi_2^{(\beta)}$ does not change the $x_1$-degree of a given monomial, only affecting $x_2$, $x_3$, and~$\beta$. This means that all monomials in $\pi_{c_2}^{(\beta)} x^\mu$ will have different tridegrees over $(\beta,x_1,x_2)$, and so on.
\end{proof}

We need one more definition concerning monomials.

\begin{definition}
	Let $\lambda = (\lambda_1, \lambda_2, \ldots, \lambda_n) $ be a partition, i.e. $\lambda_1 \geq \lambda_2 \geq \ldots \geq \lambda_n$. We shall say that monomial $x^\mu = x_1^{\mu_1} \ldots x_n^{\mu_n}$, where $\mu = (\mu_1, \mu_2, \ldots, \mu_n)$, is \emph{$\lambda$-alternating}, if $\lambda_1 \geq \mu_1 \geq \lambda_2 \geq \mu_2 \ldots \geq \lambda_{n-1} \geq \mu_{n-1} \geq \lambda_n$, and \emph{$\lambda$-nonalternating} otherwise. Denote the sum of all nonalternating monomials of a polynomial $p(x_1,\dots,x_n)$ by $[p]_\lambda$.
\end{definition}

\begin{lemma}\label{lem:55} Let $\lambda = (\lambda_1 \geq \ldots \geq \lambda_n)$ be a partition. Then we have
\[
[\pi_{c_{n-1}}^{(\beta)}x^\lambda]_\lambda\in\Ker \pi_{c_1\dots c_{n-2}}^{(\beta)}.
\]
\end{lemma}

Before giving the general proof of this lemma, we consider an example for $n=3$.

\begin{example}\label{ex:51} Let $n = 3$. Take a partition $\lambda=(a,b,0)$ and apply $\pi_{s_2s_1}^{(\beta)}$  to $x^\lambda=x_1^ax_2^b$. The resulting monomials are shown in Figure~\ref{fig:altnonalt}, with all $\lambda$-alternating and nonalternating monomials  located in the blue and green area, respectively. In this figure, we have the case $\lambda=(4,2,0)$.

Now take the rightmost column in the green triangle. The sum of its monomials is equal to
\begin{multline*}
\beta x_1^ax_2^{b+1}+x_1^{a-1}x_2^{b+1}+\dots+\beta x_1^{b+1}x_2^{a}+x_1^bx_2^a=\\(x_1^ax_2^b+\beta x_1^ax_2^{b+1}+x_1^{a-1}x_2^{b+1}+\dots+\beta x_1^{b+1}x_2^{a}+x_1^bx_2^a)-x_1^ax_2^b=\\\pi_1^{(\beta)}(x_1^ax_2^b)-x_1^ax_2^b.
\end{multline*}
Since $\pi_1^{(\beta)}$ is idempotent, we have $\pi_1^{(\beta)}(\pi_1^{(\beta)}(x_1^ax_2^b)=\pi_1^{(\beta)}(x_1^ax_2^b)$, so $\pi_1^{(\beta)}$ applied to LHS of the equality above is zero. 

The same holds for every other column of the green triangle: the sum of nonalternating monomials in  it is equal either  to $\pi_1^{(\beta)}(x_1^{a-i}x_2^bx_3^i)-x_1^{a-i}x_2^bx_3^i$, for $i=0,\dots,a-b$, or to $\pi_1^{(\beta)}(x_1^{a-i}x_2^bx_3^{i+1})-x_1^{a-i}x_2^bx_3^{i+1}$, for $i=0,\dots,1,a-b-1$. The monomials of the first kind correspond to vertices on the blue diagonal, while the monomials of the second kind correspond to edges between them. So the sum of all the nonalternating monomials in $\pi_{s_2s_1}^{(\beta)}x^\lambda$ belongs to the kernel of $\pi_1^{(\beta)}$.

\begin{figure}[h!]
\begin{tikzpicture}[scale=1.7]
	\filldraw[blue!5] [rounded corners=10pt] (-0.4,-0.2)--(4.2,-0.2)--(8.4,4.2)--(3.8,4.2)--cycle;
	\filldraw[teal!10!white] [rounded corners=10pt] (4.3,-0.2)--(8.3,-0.2)--(8.3,4)--cycle;
	\draw (0,0)--(8,0)--(8,4)--(4,4)--cycle;
	\draw (2,0)--(6,4);
	\draw [NavyBlue, very thick](4,0)--(8,4);
	\draw (2,2)--(8,2);
	\draw (6,0)--(6,2);
	\filldraw (8,4) node [NavyBlue, above right] {$x_1^4x_2^2$} circle (2pt); 
	\filldraw (8,2) node [right] {$x_1^3x_2^3$} circle (2pt); 	
	\filldraw (8,0) node [below right] {$x_1^2x_2^4$} circle (2pt); 
	\filldraw (6,0) node [below] {$x_1^2x_2^3x_3$} circle (2pt); 
	\filldraw (4,0) node [NavyBlue, below] {$x_1^2x_2^2x_3^2$} circle (2pt); 
	\filldraw (2,0) node [below] {$x_1^2x_2x_3^3$} circle (2pt); 
	\filldraw (0,0) node [below left] {$x_1^2x_3^4$} circle (2pt); 

	\filldraw (6,2) node [NavyBlue, above left] {$x_1^3x_2^2x_3$} circle (2pt); 
	\filldraw (4,2) node [above left] {$x_1^3x_2x_3^2$} circle (2pt); 
	\filldraw (2,2) node [left] {$x_1^3x_3^3$} circle (2pt); 
	\filldraw (6,4) node [above] {$x_1^4x_2x_3$} circle (2pt); 
	\filldraw (4,4) node [above left] {$x_1^4x_3^2$} circle (2pt); 

	\draw (7.3,3.3) node [NavyBlue, above left,rotate=45] {$\beta x_1^4x_2^2x_3$};
	\draw (5.3,1.3) node [NavyBlue, above left,rotate=45] {$\beta x_1^3x_2^2x_3^2$};
	
	\draw (5.3,3.3) node [above left,rotate=45] {$\beta x_1^4x_2x_3^2$};
	\draw (3.3,1.3) node [above left,rotate=45] {$\beta x_1^3x_2x_3^3$};
	\draw (3.3,3.3) node [above left,rotate=45] {$\beta x_1^4x_3^3$};
	\draw (1.3,1.3) node [above left,rotate=45] {$\beta x_1^3x_3^4$};
	\draw (8,3) node [below,rotate=90] {$\beta x_1^4x_2^3$};
	\draw (8,1) node [below,rotate=90] {$\beta x_1^3x_2^4$};
	\draw (6,1) node [below,rotate=90] {$\beta x_1^3x_2^3x_3$};
	\draw (7,0) node [above] {$\beta x_1^2x_2^4x_3$};
	\draw (5,0) node [above] {$\beta x_1^2x_2^3x_3^2$};
	\draw (3,0) node [above] {$\beta x_1^2x_2^2x_3^3$};
	\draw (1,0) node [above] {$\beta x_1^2x_2x_3^4$};
	\draw (7,2) node [below] {$\beta x_1^3x_2^3x_3$};
	\draw (5,2) node [below] {$\beta x_1^3x_2^2x_3^2$};
	\draw (3,2) node [below] {$\beta x_1^3x_2x_3^3$};
	\draw (7,4) node [below] {$\beta x_1^4x_2^2x_3$};
	\draw (5,4) node [below] {$\beta x_1^4x_2x_3^2$};

	\draw (2,1) node [rotate=26] {$\beta^2 x_1^3x_2x_3^4$};
	\draw (4,1) node [rotate=26] {$\beta^2 x_1^3x_2^2x_3^3$};
	\draw (4,3) node [rotate=26] {$\beta^2 x_1^4x_2x_3^3$};
	\draw (6,3) node [rotate=26] {$\beta^2 x_1^4x_2^2x_3^2$};
	
	\draw (7.4,2.6) node [rotate=45] {$\beta^2 x_1^4x_2^3x_3$};

	\draw (7,1) node [rotate=45] {$\beta^2 x_1^3x_2^4x_3$};
	\draw (5.5,0.9) node [rotate=45] {$\beta^2 x_1^3x_2^3x_3^2$};
\end{tikzpicture}
\caption{Alternating and nonalternating monomials}\label{fig:altnonalt}	
\end{figure}
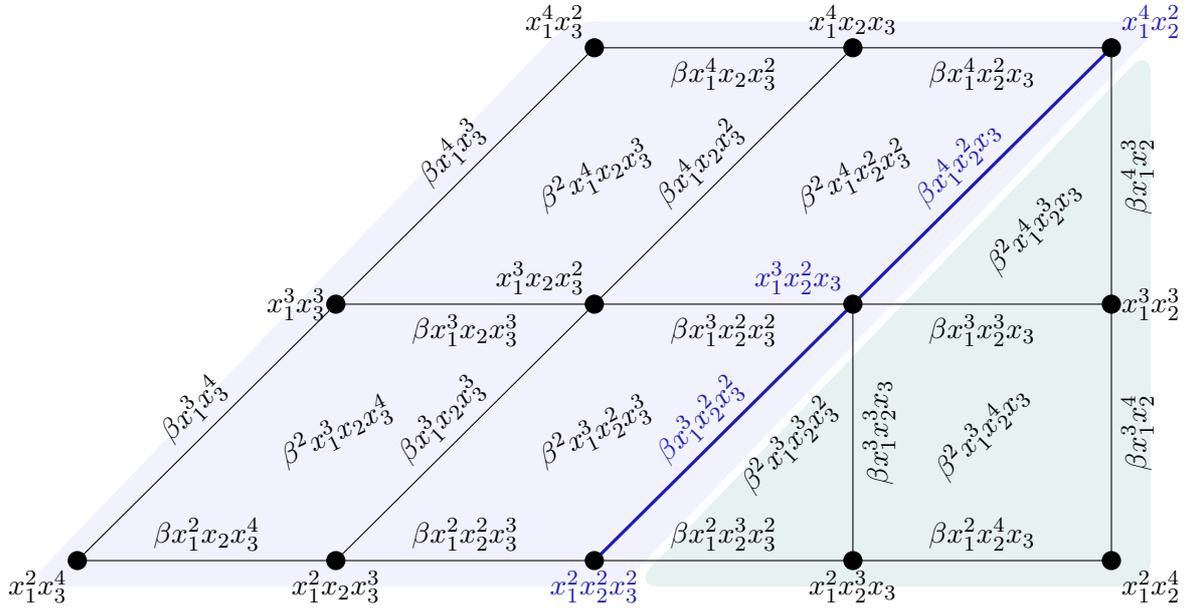

Moreover, the sum of the nonalternating monomials in every column is equal to  $\pi_{1}^{(\beta)}(x_1^{a - c} x_2^bx_3^c) - x_1^{a - c} x_2^bx_3^c$ or $\pi_{1}^{(\beta)}(\beta x_1^{a - c} x_2^bx_3^{c+1}) - \beta x_1^{a - c} x_2^bx_3^{c+1}$ respectively. Recall that the monomials $x_1^{a-c}x_2^bx_3^c$ correspond to vertices on the diagonal, while the monomials $\beta x_1^{a-c}x_2^bx_3^{c+1}$ correspond to edges between them (cf.~Sec.~\ref{ssec:ex-general}).  This means that $\pi_{1}^{(\beta)}([\pi_{s_2s_1}^{(\beta)}x^\lambda]_\lambda)=0$: the sum of all nonalternating monomials  belongs to the kernel of $\pi_1^{(\beta)}$. 
\end{example}

\begin{proof}[Proof of Lemma~\ref{lem:55}] The proof is by induction on $n$.
For $n = 1$ or $2$, there is nothing to prove. The case $n = 3$ was considered in Example~\ref{ex:51}.

The general case is treated similarly to the case $n=3$.
Consider a monomial $x^\lambda$ and apply to it $\pi_2^{(\beta)}\pi_1^{(\beta)}$. This is the sum of monomials $x_1^{\nu_1} x_2^{\nu_2} x_3^{\nu_3} x_4^{\lambda_4}\dots x_n^{\lambda_n}$, where $\nu_3\geq\lambda_4\geq\dots$ is a partition. Denote by $A$ the sum of all monomials satisfying the condition $\lambda_2\geq \nu_2\geq \lambda_3$. (They correspond to the blue area in Figure~\ref{fig:altnonalt}). In Example~\ref{ex:51}, these were alternating monomials; here these are monomials that satisfy the alternating condition for $\nu_1$ and $\nu_2$.
The remaining monomials (those in the green area) can be represented as
\[
\sum_{j=0}^{\lambda_1-\lambda_2}\left( \pi_1^{(\beta)}(x_1^{\lambda_1-j}x_2^{\lambda_2}x_3^{\lambda_3+j}x_4^{\lambda_4}\dots)- x_1^{\lambda_1-j}x_2^{\lambda_2}x_3^{\lambda_3+j}x_4^{\lambda_4}\dots\right),
\]
where monomials $x_1^{\lambda_1-j}x_2^{\lambda_2}x_3^{\lambda_3+j}x_4^{\lambda_4}\dots$ are located on the blue diagonal (see Figure~\ref{fig:altnonalt}). This implies that the sums $\left( \pi_1^{(\beta)}(x_1^{\lambda_1-j}x_2^{\lambda_2}x_3^{\lambda_3+j}x_4^{\lambda_4}\dots)- x_1^{\lambda_1-j}x_2^{\lambda_2}x_3^{\lambda_3+j}x_4^{\lambda_4}\dots\right)$ are located in the corresponding columns in the green area. The sum of all such monomials will be denoted by $B$. Then $\pi_2^{(\beta)}\pi_1^{(\beta)}(x_\lambda) = A + B$.

Now denote $\pi_{n-1}^{(\beta)}\dots \pi_3^{(\beta)}$ by $\pi$. We have
\[
\pi_{c_{n-1}}^{(\beta)}x_\lambda=(\pi_{n-1}^{(\beta)}\dots \pi_3^{(\beta)})(\pi_2^{(\beta)}\pi_1^{(\beta)}(x_\lambda)) = \pi (A + B).
\]
 Since $\pi$ is a linear operator, we have
 \[
 \pi (A + B) = \pi (A) + \pi (B).
 \]
Monomials in $x_3, x_4, \ldots$ are symmetric in $x_1$ and $x_2$, hence
\begin{multline*}
 \pi_1^{(\beta)}(x_1^{\lambda_1-j}x_2^{\lambda_2}x_3^{\lambda_3+j}x_4^{\lambda_4}\dots)- x_1^{\lambda_1-j}x_2^{\lambda_2}x_3^{\lambda_3+j}x_4^{\lambda_4}\dots = \\
 =(x_3^{\lambda_3+j}x_4^{\lambda_4}\dots )\cdot \pi_1^{(\beta)}(x_1^{\lambda_1-j}x_2^{\lambda_2})
  -x_1^{\lambda_1-j}x_2^{\lambda_2}x_3^{\lambda_3+j}x_4^{\lambda_4}\dots = \\
  =(x_3^{\lambda_3+j}x_4^{\lambda_4}\dots) (\pi_1^{(\beta)} (x_1^{\lambda_1-j}x_2^{\lambda_2}) - x_1^{\lambda_1-j}x_2^{\lambda_2}).
\end{multline*}
Since $(\pi_1^{(\beta)} (x_1^{\lambda_1-j}x_2^{\lambda_2}) - x_1^{\lambda_1-j}x_2^{\lambda_2})$  is symmetric in $x_3, x_4, \ldots$, we have:
\[
\pi \left( (x_3^{\lambda_3+j}x_4^{\lambda_4}\dots) (\pi_1^{(\beta)} (x_1^{\lambda_1-j}x_2^{\lambda_2}) - x_1^{\lambda_1-j}x_2^{\lambda_2})\right) = \pi (x_3^{\lambda_3+j}x_4^{\lambda_4}\dots) \cdot (\pi_1^{(\beta)} (x_1^{\lambda_1-j}x_2^{\lambda_2}) - x_1^{\lambda_1-j}x_2^{\lambda_2}).
\]
Let us introduce some extra notations. Denote by $B_0^{j}$  and $B_1^{j}$ the sum of all alternating and nonalternating monomials in $\pi (x_3^{\lambda_3+j}x_4^{\lambda_4}\dots)$, respectively. Then $\pi (x_3^{\lambda_3+j}x_4^{\lambda_4}\dots) = B_0^{j} + B_1^{j}$. Note that $B_0^{j}$ and $ B_1^{j}$ are elements of $\mathbb{Z}[x_3, x_4, \ldots]$. Also we denote by $A_0$ and $A_1$  the sums of all alternating and nonalternating monomials in $\pi (A)$, respectively.
Then we have:
\[
\pi_{c_{n-1}}^{(\beta)}x_\lambda= \pi(A) + \pi (B) = A_0 + A_1 + \sum_{j=0}^{\lambda_1-\lambda_2}\left( (B_0^{j} + B_1^{j}) (\pi_1^{(\beta)} (x_1^{\lambda_1-j}x_2^{\lambda_2}) - x_1^{\lambda_1-j}x_2^{\lambda_2}) \right)
\]

To prove this lemma for an arbitrary $n$, we observe that $c_1\dots c_{n-2}=w_0'$ is the longest permutation for the subgroup $\langle s_1,\dots,s_{n-2}\rangle \cong S_{n-1}\subset S_n$ and fix another word for  $w_0'$. We denote by $w_0'' $ the longest permutation in the subgroup $\langle s_2,\dots, s_{n-2}\rangle\cong S_{n-2} \hookrightarrow S_{n-1}\hookrightarrow S_n$. Then we have
\begin{equation}
w'_0 = (s_{n-2} \ldots s_2 s_1) \cdot w_0''=s_1(s_2s_1)\dots(s_{n-2}\dots s_1).\label{eq:w0prime}	
\end{equation}

Consider the polynomial $\pi_{c_{n-1}}^{(\beta)}x^\lambda$. According to Lemma~\ref{lem:multfree}, it is multiplicity free. Let $x^\mu$ be a $\lambda$-nonalternating monomial occurring in it. By construction, we have $\lambda_1\geq \mu_1$ and $\mu_i\geq \lambda_{i+1}$ for each $1\leq i\leq n-1$. We distinguish between the two cases:

(1) There exists a $k$ such that $\mu_k> \lambda_k$, with $k\geq 3$. Denote the sum of all such nonalternating monomials from $[\pi_{c_{n-1}}^{(\beta)}x^\lambda]_\lambda$ by $r_1(\beta, x_1,\dots,x_n)$. With the previous notation:
\[
r_1(\beta, x_1,\dots,x_n) = A_1 + \sum_{j=0}^{\lambda_1-\lambda_2}\left(B_1^{j} \cdot (\pi_1^{(\beta)} (x_1^{\lambda_1-j}x_2^{\lambda_2}) - x_1^{\lambda_1-j}x_2^{\lambda_2}) \right)
\]

(2) We have $\mu_2>\lambda_2$ and $\lambda_i\geq \mu_i$ for each $i\geq 3$. The sum of all such nonalternating monomials from $[\pi_{c_{n-1}}^{(\beta)}x^\lambda]_\lambda$ will be denoted by $r_2(\beta,x_1,\dots,x_n)$. With the previous notation:

\[
r_2(\beta,x_1,\dots,x_n) = \sum_{j=0}^{\lambda_1-\lambda_2}\left(B_0^{j} \cdot (\pi_1^{(\beta)} (x_1^{\lambda_1-j}x_2^{\lambda_2}) - x_1^{\lambda_1-j}x_2^{\lambda_2}) \right)
\]

Now we shall check that $\pi_{w_0'}^{(\beta)}$ applied to each of $r_1$ and $r_2$ equals zero. For this, we take different words for $w_0'$. For the first one, take 
\[
w_0'=(s_{n-2} \ldots s_2 s_1) \cdot w_0'',
\]
where $w_0''=  s_2  (s_3s_2)  \ldots (s_{n-1}s_{n-2} \ldots s_2)$ is the longest permutation in the subgroup $\langle s_2,\dots, s_{n-2}\rangle\cong S_{n-2} \hookrightarrow S_{n-1}$. Then $\pi_{w_0''}^{(\beta)} r_1=0$ by the induction hypothesis, and so is $\pi_{w_0'}^{(\beta)}$.

In the second case, we claim that $\pi_1^{(\beta)}$ annihilates $r_2$, so does $\pi_{w_0'}^{(\beta)}$. Indeed, since the word~(\ref{eq:w0prime}) ends with $s_1$, we have:
\[
\pi_1^{(\beta)}(r_2) =  \pi_1^{(\beta)} \left( \sum_{j=0}^{\lambda_1-\lambda_2}\left(B_0^{j} \cdot (\pi_1^{(\beta)} (x_1^{\lambda_1-j}x_2^{\lambda_2}) - x_1^{\lambda_1-j}x_2^{\lambda_2}) \right) \right) = 
\]

\[
= \sum_{j=0}^{\lambda_1-\lambda_2} \pi_1^{(\beta)} \left( B_0^{j} \cdot (\pi_1^{(\beta)} (x_1^{\lambda_1-j}x_2^{\lambda_2}) - x_1^{\lambda_1-j}x_2^{\lambda_2}) \right) = \sum_{j=0}^{\lambda_1-\lambda_2} B_0^{j} \cdot  \pi_1^{(\beta)}  ( (\pi_1^{(\beta)} (x_1^{\lambda_1-j}x_2^{\lambda_2}) - x_1^{\lambda_1-j}x_2^{\lambda_2}) )
\]
Recall that $\pi_i^{(\beta)} ( \pi_i^{(\beta)} f - f) = 0$. It follows that:
\[
\pi_1^{(\beta)}(r_2) = \sum_{j=0}^{\lambda_1-\lambda_2} B_0^{j} \cdot  \pi_1^{(\beta)}  ( (\pi_1^{(\beta)} (x_1^{\lambda_1-j}x_2^{\lambda_2}) - x_1^{\lambda_1-j}x_2^{\lambda_2}) ) = \sum_{j=0}^{\lambda_1-\lambda_2} B_0^{j} \cdot 0 = 0.
\]

\end{proof}

Now let us act on $x^\lambda$ by $\pi^{(\beta)}_{w_0}$, where we take the word for $w_0$ given by~(\ref{eq:w0}). Consider the following sequences of monomials:
\begin{eqnarray*}
	x^{\lambda}=x^{\lambda^{(1,0)}}\xrightarrow{\pi_{1}^{(\beta)}} x^{\lambda^{(1,1)}}\xrightarrow{\pi_{2}^{(\beta)}} x^{\lambda^{(1,2)}}\xrightarrow{\pi_{3}^{(\beta)}} \dots \xrightarrow{\pi_{n-1}^{(\beta)}} x^{\lambda^{(1,n-1)}}= x^{\lambda^{(2,0)}}; \\
x^{\lambda^{(2,0)}}\xrightarrow{\pi_{1}^{(\beta)}} x^{\lambda^{(2,1)}}\xrightarrow{\pi_{2}^{(\beta)}} x^{\lambda^{(2,2)}}\xrightarrow{\pi_{3}^{(\beta)}} \dots \xrightarrow{\pi_{n-2}^{(\beta)}} x^{\lambda^{(2,n-2)}}=x^{\lambda^{(3,0)}}; \\
x^{\lambda^{(3,0)}}\xrightarrow{\pi_{1}^{(\beta)}} x^{\lambda^{(3,1)}}\xrightarrow{\pi_{2}^{(\beta)}} x^{\lambda^{(3,2)}}\xrightarrow{\pi_{3}^{(\beta)}} \dots \xrightarrow{\pi_{n-3}^{(\beta)}} x^{\lambda^{(3,n-3)}}\xrightarrow{\pi_{1}^{(\beta)}}x^{\lambda^{(4,0)}};\\
\dots\\
x^{\lambda^{(n-1,0)}} \dots \xrightarrow{\pi_1^{(\beta)}} x^{\lambda^{(n-1,1)}}\xrightarrow{\pi_{2}^{(\beta)}} x^{\lambda^{(n-1,2)}}=x^{(n,0)};\\
x^{\lambda^{(n,0)}}\xrightarrow{\pi_1^{(\beta)}} x^{\lambda^{(n,1)}}=x^\nu.
\end{eqnarray*}
Here each $x^{\lambda^{(i,j)}}$ occurs as a summand in $\pi_j^{\beta} x^{\lambda^{(i,j-1)}}$. Such a sequence is called a \emph{track} of $x^{\lambda}$. 

For each track with a nonzero $x^\nu$, we construct an efficient enhanced pattern $P$ as follows.  First, note that by Lemma~\ref{lem:55}, each $\lambda^{(i,0)}$ is ${\lambda^{(i-1,0)}}$-alternating; in particular, it is a partition. Take a pattern $P$ with entries $p_{ij}$ such that its $i$-th row contains the first $n-i$ parts of $\lambda^{(i-1,0)}$, written in the decreasing order (the initial row contains the parts of $\lambda$ and has number $0$), that is, $p_{ij}=(\lambda^{(i-1,0)})_{n-i-j+1}$. Next, draw a circle around a vertex if the degree of $\beta$ in front of the corresponding monomials $x^{\lambda^{(i,j-1)}}$ and $x^{\lambda^{(i,j)}}$ is equal.

This gives us the location of encircled vertices. According to Lemma~\ref{lem:redundant}, this defines an efficient enhanced pattern. It is clear that all the efficient patterns can be obtained in such a way, and different patterns correspond to different tracks. Theorem~\ref{thm:main1} is proved.

%the degrees of monomials $x^{\lambda^{(i,j-1)}}$ and $x^{\lambda^{(i,j)}}$ can only differ in positions $j$ and $j+1$; the degree of $\beta$ in front of these two monomials can differ by at most one.

%First, the entries in its $k$-th row of $P$ are equal to $\lambda^{(k)}_{n-k},\dots,\lambda^{(k)}_1$, i.e. to the first $n-k$ entries of $\lambda^{(k)}$ written from the right to the left. Lemmas~\ref{lem:multfree} and~\ref{lem:55} guarantee that the intermittence condition holds, so this tableau is indeed a Gelfand--Zetlin pattern.

 %Next, the entry $\lambda^{(k)}_{i}$ is encircled if the degree of $\beta$ occurring in $x^{\lambda^{(k)}_i}$ equals the degree of $\beta$ occurring in the previous monomial of the track. \todo{Edges?} This gives us an efficient enhanced pattern, with different patterns obviously corresponding to different tracks. 

%\todo[inline]{All patterns can be obtained. Monomials?}

\begin{example}\label{ex:52} Let $\lambda = (2, 1, 0)$. Below is a track of monomial $\beta^2x_1^2x_2x_3^2$ and the corresponding enhanced Gelfand--Zetlin pattern, being filled row by row from the right to the left, from top to bottom.
\[
x_1^2x_2\xrightarrow{\pi_1^{\beta}} \beta x_1^2x_2^2\xrightarrow{\pi_2^{\beta}} \beta^2x_1^2x_2x_3^2
\xrightarrow{\pi_1^{\beta}} \beta^2 x_1^2x_2x_3^2.
\]	

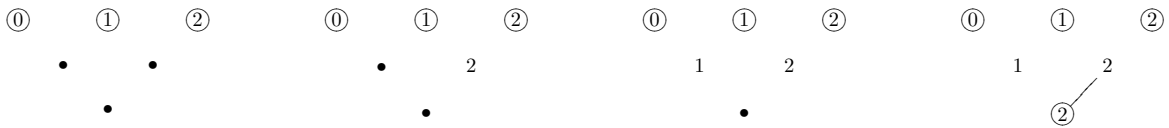
\[
\xymatrixrowsep{1pc}\xymatrixcolsep{1pc}
\scalebox{0.7}
{\xymatrix{
*+[o][F]{0} && *+[o][F]{1} && *+[o][F]{2}\\
& \bullet && \bullet \\
&&\bullet}
}\qquad\qquad
\scalebox{0.7}
{\xymatrix{
*+[o][F]{0} && *+[o][F]{1} && *+[o][F]{2}\\
& \bullet && 2 \\
&&\bullet}
}\qquad\qquad
\scalebox{0.7}
{\xymatrix{
*+[o][F]{0} && *+[o][F]{1} && *+[o][F]{2}\\
& 1 && 2 \\
&&\bullet}
}\qquad\qquad
\scalebox{0.7}
{\xymatrix{
*+[o][F]{0} && *+[o][F]{1} && *+[o][F]{2}\\
& 1 && 2 \\
&&*+[o][F]{2}\ar@{-}[ur]}
}
\]
\end{example}

\section{Proof for the general case}\label{sec:proofgen}

\subsection{Face diagrams}\label{ssec:61}
For every permutation $u \in S_n$ there exists a dual Kogan face $F$ such that $w(F) = u$. Moreover, there can be more than one such Kogan face; they correspond to reduced subwords in
\[
\mathbf{w}_0=(s_{n-1},\dots,s_{1}, s_{n-2},\dots,s_{1},\dots, s_1,s_2,s_1)
\]
with the product equal to $u$.

In this section we show how different diagrams of dual Kogan faces with the same permutation are related to each other. This is done in Lemma~\ref{lem:61} and Lemma~\ref{lem:62}. As we mentioned before, (reduced) dual Kogan faces bijectively correspond to (reduced) pipe dreams; in terms of pipe dreams these lemmas are proved in~\cite{BergeronBilley93}, but we still provide their proofs for for the sake of completeness of exposition. Then we use them  in \S~\ref{ssec:62} to prove Theorem~\ref{thm:main2}.

\subsection{Three lemmas about dual Kogan faces}

\begin{lemma}\label{lem:61} Let $F$ and $G$ be two dual Kogan faces corresponding to permutations $w_F$ and $w_G$ respectively. If their diagrams are obtained one from another by moving one edge as shown in figure below, then we have $w_F = w_G$.
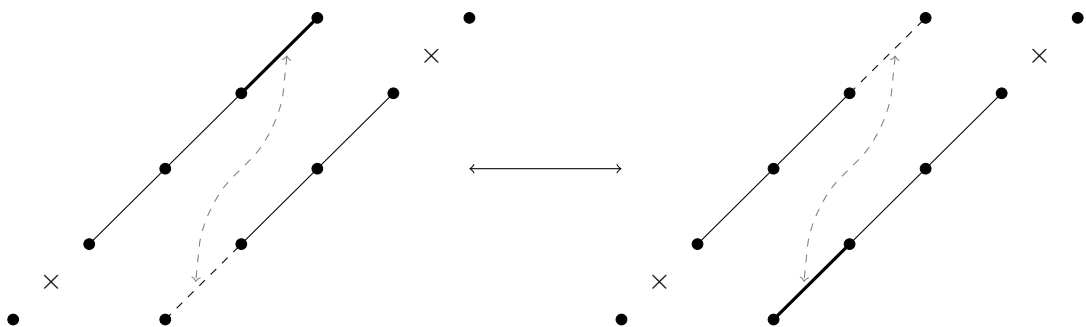
$$
\begin{tikzpicture}

\draw [very thick] (3,3) -- (4,4);
\draw [very thick] (10,0) -- (11,1);

\foreach \x in {0, 1, 2, 3, 4} {
\filldraw (\x,\x) circle (2pt);
\filldraw (\x + 2,\x) circle (2pt);
}

\draw (1,1) -- (3,3);

\draw (3, 1) -- (5, 3);
\draw [dashed] (2, 0) -- (3, 1);
\draw (0.5, 0.5) node {$\times$};
\draw (5.5, 3.5) node {$\times$};
\draw [<->, rounded corners=15pt, gray, dashed] (3.6, 3.5) -- (3.5, 2.5) -- (2.5, 1.5) -- (2.4, 0.5);

 \foreach \x in {0, 1, 2, 3, 4} {
\filldraw (\x + 8,\x) circle (2pt);
\filldraw (\x + 10,\x) circle (2pt);
 }

\draw (9,1) -- (11,3);
\draw (11, 1) -- (13, 3);

\draw [dashed] (11, 3) -- (12, 4);
\draw (8.5, 0.5) node {$\times$};
\draw (13.5, 3.5) node {$\times$};
\draw [<->, rounded corners=15pt, gray, dashed] (11.6, 3.5) -- (11.5, 2.5) -- (10.5, 1.5) -- (10.4, 0.5);

\draw [<->] (6, 2) -- (8, 2);

\end{tikzpicture}
$$
\end{lemma}

\begin{proof} Recall that simple transpositions correspond to the edges of the diagram as shown below:

$$
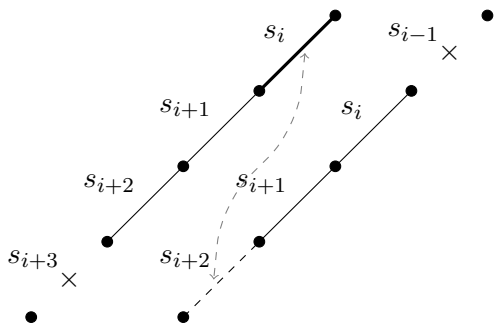
\begin{tikzpicture}

\draw [very thick] (3,3) -- node [anchor=south east] {$s_i$} (4,4);

    \foreach \x in {0, 1, 2, 3, 4} {
\filldraw (\x,\x) circle (2pt);
\filldraw (\x + 2,\x) circle (2pt);
}

\draw (1,1) -- node [anchor=south east] {$s_{i + 2}$}(2, 2) -- node [anchor=south east]  {$s_{i + 1}$}(3,3);

\draw (3, 1) -- node [anchor=south east] {$s_{i + 1}$} (4, 2) -- node [anchor=south east] {$s_i$} (5, 3);
\draw [dashed] (2, 0) -- node [anchor=south east] {$s_{i + 2}$} (3, 1);
\draw (0.5, 0.5) node [anchor=south east] {$s_{i + 3}$} node {$\times$};
\draw (5.5, 3.5) node [anchor=south east] {$s_{i - 1}$} node {$\times$};
\draw [<->, rounded corners=15pt, gray, dashed] (3.6, 3.5) -- (3.5, 2.5) -- (2.5, 1.5) -- (2.4, 0.5);
\end{tikzpicture}
$$

It follows that reading this diagram from bottom to top from right to left we get the word
\[ w_F = \ldots s_{i+1} s_{i+2} \ldots s_i s_{i+1} \ldots \underline{s_i} \ldots.
\]
(here we underline the letter corresponding to the edge we are moving).

Since simple transpositions satisfy braid relations, we get a word for face $G$: 
\begin{multline*}
	w_F =\ldots s_{i+1} s_{i+2} \ldots s_i s_{i+1} \ldots \underline{s_i} \ldots  = \ldots s_{i+1} s_{i+2} \ldots s_i s_{i+1} \underline{s_i} \ldots=\\
= \ldots s_{i+1} s_{i+2} \ldots \underline{s_{i+1}}  s_i s_{i+1} \ldots = \ldots \underline{s_{i+2}}s_{i+1} s_{i+2} \ldots   s_i s_{i+1} \ldots = \\
\underline{s_{i+2}}  \ldots s_{i+1} s_{i+2} \ldots   s_i s_{i+1} \ldots = w_G.
\end{multline*} 
\end{proof}

\begin{definition}\label{def:rightdiag} A diagram is called \emph{right-adjusted} if all its edges are pushed towards the right side.
\end{definition}
Note that every right-adjusted diagram corresponds to a word of the form
$$
\mathbf{u} = (\ldots,s_3,s_{4}, \ldots,s_m, s_2, s_{3}, \ldots, s_k,s_1 s_{2}, \ldots, s_r).
$$ 
Such a word is called \emph{canonical}. It is easy to see that this canonical word is reduced. It follows that there exists a unique right-adjusted diagram for every permutation.

\begin{lemma}\label{lem:62}  Let $F$ and $G$ be two dual Kogan faces such that $w(F) = w(G)$. Then the diagram for $F$ is obtained from the diagram for $G$ by applying transformations described in Lemma~\ref{lem:61}.
\end{lemma}
\begin{proof} Note that it is enough to prove this lemma for a face $G$ with the right-adjusted diagram. Suppose the diagram of $F$ is not right-adjusted. It follows that there exists an edge such that there is no edge immediately to the right of it. Let $e_0$ be the lowest rightmost edge with such a property. Since a word of simple transpositions is necessarily reduced, then we can move this edge several rows down using Lemma~\ref{lem:61}. Continuing this procedure, we will get a right-adjusted diagram.
\end{proof}

\begin{lemma}\label{lem:63} Suppose that the diagram of face $F$ is right-adjusted. Then permutation $w(F)$ is equal to the permutation $u$ obtained by reading the empty places of diagram $F$ from  bottom to top from left to right in the following way:
$$
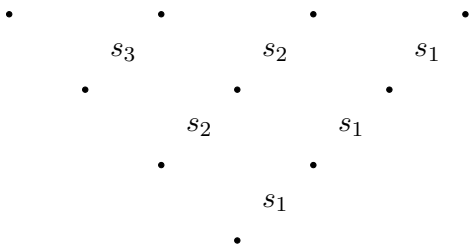
\begin{tikzpicture}

\foreach \x in {1, 2, 3, 4}
\filldraw (2*\x, 0) circle (1pt);

\foreach \x in {1, 2, 3}
\filldraw (2*\x +1, -1)  circle (1pt);

\foreach \x in {1, 2, 3}
\draw (-2*\x +9.5, -0.5) node{$s_{\x}$};

\foreach \x in {1, 2}
\filldraw (2*\x +2, -2)  circle (1pt);

\foreach \x in {1, 2}
\draw (-2*\x +8.5, -1.5) node{$s_{\x}$};

\filldraw (5, -3)  circle (1pt);
\draw (5.5, -2.5) node{$s_{1}$};
    
\end{tikzpicture}
$$
\end{lemma}

\begin{proof} First we recall that the word $u$ for the face $F$ described in \S\;\ref{ssec:gz} obtained by reading a diagram from bottom to top from right to left. The permutation $w(F)$ is equal to $w_0u$, there $w_0 \in S_n$ is the longest permutation.

The proof is by induction on $n$. For $n = 1$ there is nothing to prove.
We denote by $w_0' \in S_{n-1}$ the longest permutation, here $S_{n - 1}$ is generated by $s_2, s_3, \ldots, s_{n - 1}$. Let the word for the face $F$ be equal to $u = u' (s_1 s_2 \ldots s_k)$, with exactly $k$ edges in the first row, and let $u' \in S_{n-1}$ be obtained by reading rows starting from the second one. Then $w(F)$ is equal to $w(F) = w_0u = (s_{n-1}\ldots s_2 s_1)w_0'u'(s_1 s_2 \ldots s_k)$. The permutation $w_0'u'$ corresponds to a smaller diagram obtained by restricting our diagram on rows starting from the second. 
Graphically this is shown on Figure~\ref{fig:wires} (left).
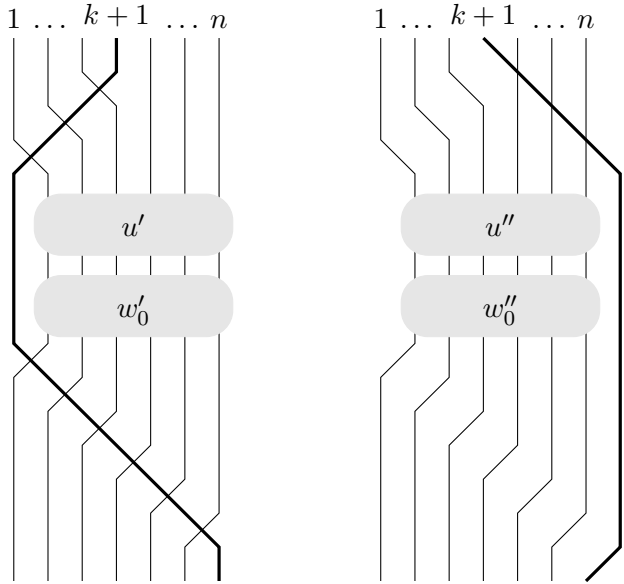
\begin{figure}[h!]

\begin{tikzpicture}[scale = 0.9]
    \draw (0,0) node [above] {$1$} -- (0, -1.5) -- (0.5, -2) -- (0.5, -4.5)-- (0, -5) -- (0, -8);
    \draw (0.5,0) -- (0.5, -1) -- (1, -1.5) -- (1, -5) -- (0.5, -5.5) -- (0.5, -8);
    \draw (1,0) node [above left] {$\dots$} -- (1, -0.5) -- (1.5, -1) -- (1.5, -5.5) -- (1, -6) -- (1, -8);
    \draw (2, 0)  -- (2, -6) -- (1.5, -6.5) -- (1.5, -8);
    \draw (2.5, 0) node [above] {$\dots$}  -- (2.5, -6.5) -- (2, -7) -- (2, -8);
    \draw (3, 0) node [above] {$n$} -- (3, -7) -- (2.5, -7.5) -- (2.5, -8);

    \filldraw[black!10!white] [rounded corners=10pt] (1.5, -2.3) -- (3.2, -2.3) -- (3.2, -3.2) -- (0.3, -3.2) -- (0.3, -2.3) -- (1.5, -2.3);

    \node at (1.75, -2.75) {$u'$};

    \filldraw[black!10!white] [rounded corners=10pt] (1.5, -3.5) -- (3.2, -3.5) -- (3.2, -4.4) -- (0.3, -4.4) -- (0.3, -3.5) -- (1.5, -3.5);

    \node at (1.75, -4) {$w_0'$};

    \draw[very thick] (1.5,0) node [above] {$k + 1$} -- (1.5, -0.5) -- (0, -2) -- (0, -4.5) -- (3, -7.5) -- (3, -8);
    
\end{tikzpicture}
\qquad\qquad 
\begin{tikzpicture}[scale = 0.9]
    \draw (0,0) node [above] {$1$} -- (0, -1.5) -- (0.5, -2) -- (0.5, -4.5)-- (0, -5) -- (0, -8);
    \draw (0.5,0) -- (0.5, -1) -- (1, -1.5) -- (1, -5) -- (0.5, -5.5) -- (0.5, -8);
    \draw (1,0) node [above left] {$\dots$} -- (1, -0.5) -- (1.5, -1) -- (1.5, -5.5) -- (1, -6) -- (1, -8);
    \draw (2, 0)  -- (2, -6) -- (1.5, -6.5) -- (1.5, -8);
    \draw (2.5, 0) node [above] {$\dots$}  -- (2.5, -6.5) -- (2, -7) -- (2, -8);
    \draw (3, 0) node [above] {$n$} -- (3, -7) -- (2.5, -7.5) -- (2.5, -8);

    \filldraw[black!10!white] [rounded corners=10pt] (1.5, -2.3) -- (3.2, -2.3) -- (3.2, -3.2) -- (0.3, -3.2) -- (0.3, -2.3) -- (1.5, -2.3);

    \node at (1.75, -2.75) {$u''$};

    \filldraw[black!10!white] [rounded corners=10pt] (1.5, -3.5) -- (3.2, -3.5) -- (3.2, -4.4) -- (0.3, -4.4) -- (0.3, -3.5) -- (1.5, -3.5);

    \node at (1.75, -4) {$w_0''$};

    \draw[very thick] (1.5,0) node [above] {$k + 1$} -- (3.5, -2) -- (3.5, -7.5) -- (3, -8);
\end{tikzpicture}
\caption{Wiring diagrams}\label{fig:wires}
\end{figure}

We denote by $u''$ and $w_0''$ permutations obtained by replacing every $s_i$ by $s_{i -1}$ in reduced words for $u'$ and $w_0'$. Since $u', w_0' \in S_{n-1}$, where $S_{n-1}$ is generated by $s_2, \ldots, s_{n-1}$,  permutations $u''$ and $w_0''$ are well defined. It follows that $w(F) = w_0u = (s_{n-1}\ldots s_2 s_1)w_0'u'(s_1 s_2 \ldots s_k) = w_0''u'' (s_{n-1} \ldots s_{k +2} s_{k +1})$ (See Figure~\ref{fig:wires}, right).

By the induction hypothesis, the permutation $w_0''u'' (s_{n-1} \ldots s_{k +2} s_{k +1})$ is obtained by reading empty places of diagram $F$ from bottom to top from right to left.
\end{proof}

This lemma will play a key role in the proof of the main result. Note that the empty places marked as shown in Lemma~\ref{lem:63} correspond to actions of Demazure--Lascoux operators described in \S~\ref{ssec:52}.

\begin{remark}
Lemma~\ref{lem:63} is a standard fact about permutations; see, for instance, \cite[R\'emarque~2.1.9]{Manivel98}.	
\end{remark}

\subsection{Proof of Theorem~\ref{thm:main2}}\label{ssec:62}

In this section we show that the restriction of the construction described in Theorem~\ref{thm:main1} to the union of dual Kogan faces gives us an arbitrary Lascoux polynomial.

The main idea of this section is to fix the canonical word for the permutation $u \in S_n$ described in Lemma~\ref{lem:63}. Then alternating monomials will be located in the face $F$ with a right-adjusted diagram. But not all nonalternating monomials will necessarily cancel. Their offspring will be located in other diagrams corresponding to the permutation $u$.

Following Lemma~\ref{lem:62} we recall that any diagram of face $F$ is obtained from the right-adjusted diagram by moving edges. Moreover, we first can move edges to the first row (from left to right), and at each step, the restriction of the diagram to rows starting from the second one will be right-adjusted.

Let us enumerate the diagonals rotated to the right-up in a triangular tableau from left to right. Note that all edges in diagrams of dual Kogan faces will be directed along such diagonals.
\begin{lemma}\label{lem:6.4}
Consider a diagram $F$, where the rows starting from the second one form a right-adjusted diagram corresponding to  permutation $u$, and the first row is filled as follows: places from $1$ to $k$ are filled by edges, places from $k + 1$ to $i > k + 1$ are empty, the remaining places can be filled in any way (see Figure~\ref{fig:tolemma}). Then  the nonalternating monomials appearing at the $i$-th step (that means, with the power of $x_i$ 
bigger than necessary) do not cancel if and only if it is possible to move an edge on the place $i$ in the first row.
\end{lemma}
\begin{figure}[h!]
\begin{tikzpicture}
    \scriptsize

    \filldraw[black!10!white] [rounded corners=10pt] (3, -0.75) -- (5.25, -0.75) -- (3, -3) -- (0.75, -0.75) -- (3, -0.75);
    
    \foreach \x in {0, 1, 2, 3, 4, 5, 6}
    \filldraw (\x, 0) circle (1.5pt);

    \foreach \x in {0, 1, 2, 3, 4, 5}
    \filldraw (\x + 0.5, -0.5) circle (1.5pt);

    \draw (6, 0) -- (5.5, -0.5);
    \draw (5, 0) -- (4.5, -0.5);

    \draw (3.75, -0.25) node {$\times$};
    \draw (2.75, -0.25) node {$\times$};
    \draw (1.75, -0.25) node {$\star$};
    \draw (0.75, -0.25) node {$\star$};
    \small
    \draw (3, 0) node [above = 1pt] {$i$};
    \draw (4, 0) node [above] {$i - 1$};
    \large
    \draw (3, -1.5) node {$u$};
\end{tikzpicture}
\caption{To Lemma~\ref{lem:6.4}}\label{fig:tolemma}
\end{figure}

\begin{proof} Following Example~\ref{ex:51}, recall that nonalternating monomials could be divided into groups of the form $\pi_{i-1}^{(\beta)}(m) - m$, there $m$ is an alternating monomial with the maximal allowed power of $x_i$ (located on the diagonal in Figure~\ref{fig:altnonalt2}.
\begin{figure}[h!]
\begin{tikzpicture}[rotate = 180]
\footnotesize
    \filldraw[teal!10!white] [rounded corners=10pt](4.8, 2) -- (4.8, 3.2) -- (7.8, 3.2) -- (4.8, 0.1) -- (4.8, 2);

    \draw (5, 0) -- (5,3) -- (10,3) -- (7,0) -- (5,0);
    \draw (5,0) -- (8,3);
    \draw (6,0) -- (9,3);
    \draw (5,1) -- (8,1);
    \draw (5,2) -- (9,2);
    \draw (6,1) -- (6,3);
    \draw (7,2) -- (7,3);
    \draw [NavyBlue, very thick] (5,0) -- (8,3);

    \filldraw [NavyBlue](5,0) node [right] {$x_{i-1}^a x_{i}^b$} circle (2pt);
    \filldraw (5,1) circle (2pt);
    \filldraw  (5,2) circle (2pt);
    \filldraw  (5,3) node [right] {$x_{i - 1}^b x_{i}^a$} circle (2pt);
    \draw (4.5, 1.5) node [rotate = 90] {$\ldots$};

    \filldraw  [NavyBlue](8,3) node [below] {$x_{i - 1}^b x_{i}^b x_{i+1}^{a-b}$} circle (2pt);
    \filldraw  (10,3) node [below] {$x_{i - 1}^b x_{i + 1}^a$} circle (2pt);
    \filldraw  (7,0) node [above] {$x_{i - 1}^a x_{i + 1}^b$} circle (2pt);
\end{tikzpicture}
\caption{Alternating and nonalternating monomials, general case}\label{fig:altnonalt2}
\end{figure}
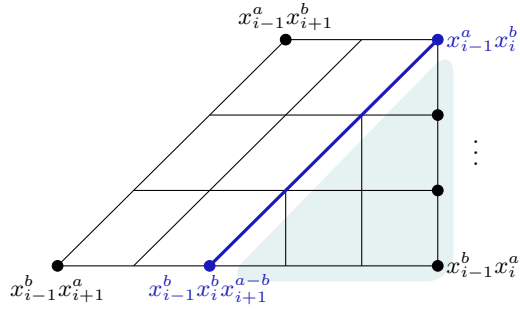

Such monomials vanish under the action of  operator $\pi_{i-1}^{(\beta)}$. Since $\pi_{i-1}^{(\beta)}$ commutes with operators $\pi_{i+ 1}^{(\beta)}, \ldots, \pi_{n-1}^{(\beta)}$ (that correspond to the last part of the first row), we should check whether a permutation $u$ can start with a transposition $s_{i-1}$.

Recall that $u = \ldots \cdot s_{n-3} \ldots s_{k_2} \cdot s_{n-2} \ldots s_{k_1}$. In braid terms this means that first the $k_1$-th strand is going down to the position $n-1$, then the $k_2$-th strand is going down  to the position $n-2$, and so on. Then we can put transposition $s_{i-1}$ on the first place if and only if the $(i-1)$-th and $i$-th strands intersect each other or, equivalently, the $(i-1)$-th line goes to the end earlier than $i$-th line or, equivalently, there is no edge in the corresponding place in the diagram. On the other hand, if $(i-1)$-th and $i$-th strands do not intersect, we can move the edge to the first row and intersect lines. Suppose the new diagram corresponds to permutation $u'$. It is easy to see that $\pi_u^{(\beta)}(\pi_{i-1}^{(\beta)}(m))$ coincides with $\pi_{u'}^{(\beta)}(m)$. It follows that the offsprings of $\pi_{i-1}^{(\beta)}(m)$ will be located in diagram $u'$. Since in this case we consider monomial $m$, we should add an edge to the $i$-th place in the first row.

$$
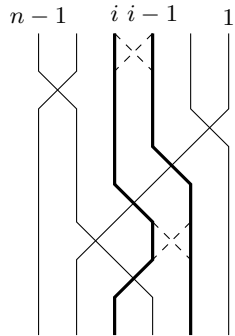
\begin{tikzpicture}
\footnotesize
    \draw (0,0) node [above] {$n-1$} -- (0, -0.5) -- (0.5, -1) -- (0.5, -2.5) -- (1.5, -3.5) -- (1.5, -4);
    \draw (0.5, 0) -- (0.5, -0.5) -- (0, -1) -- (0, -4);
    \draw (2, 0) -- (2, -1) -- (2.5, -1.5) -- (2.5, -4);
    \draw (2.5, 0) node [above] {$1$} -- (2.5, -1) -- (0.5, -3) -- (0.5, -4);
    \draw [very thick] (1, 0)  node [above = 1pt] {$i$} -- (1, -2) -- (1.5, -2.5) -- (1.5, -3) -- (1, -3.5) -- (1, -4);
    \draw [very thick] (1.5, 0) node [above] {$i-1$} -- (1.5, -1.5) -- (2, -2) -- (2, -4);

    \draw [dashed] (1.5, -2.5) -- (2, -3);
    \draw [dashed] (2, -2.5) -- (1.5, -3);

    \draw [dashed] (1, 0) -- (1.5, -0.5);
    \draw [dashed] (1.5, 0) -- (1, -0.5);

\end{tikzpicture}
$$
\end{proof}

Proof of Theorem~\ref{thm:main2} is obtained by filling the diagram row by row from top to bottom from left to right and applying Lemma~\ref{lem:6.4} at each step.

%\bibliographystyle{alpha}
%\bibliography{lascoux}

\begin{thebibliography}{BSW20}

\bibitem[And85]{Andersen85}
Henning~Haahr Andersen.
\newblock Schubert varieties and {D}emazure's character formula.
\newblock {\em Invent. Math.}, 79(3):611--618, 1985.

\bibitem[BB93]{BergeronBilley93}
Nantel Bergeron and Sara Billey.
\newblock R{C}-graphs and {S}chubert polynomials.
\newblock {\em Experiment. Math.}, 2(4):257--269, 1993.

\bibitem[BSW20]{BSW20}
Valentin Buciumas, Travis Scrimshaw, and Katherine Weber.
\newblock Colored five-vertex models and {L}ascoux polynomials and atoms.
\newblock {\em J. Lond. Math. Soc. (2)}, 102(3):1047--1066, 2020.

\bibitem[Buc02]{Buch02}
Anders~Skovsted Buch.
\newblock A {L}ittlewood-{R}ichardson rule for the {$K$}-theory of
  {G}rassmannians.
\newblock {\em Acta Math.}, 189(1):37--78, 2002.

\bibitem[Dem74]{Demazure74}
Michel Demazure.
\newblock Une nouvelle formule des caract\`eres.
\newblock {\em Bull. Sci. Math. (2)}, 98(3):163--172, 1974.

\bibitem[FK94]{FominKirillov94}
Sergey Fomin and Anatol~N. Kirillov.
\newblock Grothendieck polynomials and the {Y}ang-{B}axter equation.
\newblock In {\em Proc. Formal Power Series and Alg. Comb}, pages 183--190,
  1994.

\bibitem[FK96]{FominKirillov96}
Sergey Fomin and Anatol~N. Kirillov.
\newblock The {Y}ang-{B}axter equation, symmetric functions, and {S}chubert
  polynomials.
\newblock In {\em Proceedings of the 5th {C}onference on {F}ormal {P}ower
  {S}eries and {A}lgebraic {C}ombinatorics ({F}lorence, 1993)}, volume 153,
  pages 123--143, 1996.

\bibitem[GC50]{GelfandZetlin50}
Israel~M. Gelfand and Mikhail~L. Cetlin.
\newblock Finite-dimensional representations of the group of unimodular
  matrices.
\newblock {\em Doklady Akad. Nauk SSSR (N.S.)}, 71:825--828, 1950.

\bibitem[GL96]{GonciuleaLakshmibai96}
Nicolae Gonciulea and V.~Lakshmibai.
\newblock Degenerations of flag and {S}chubert varieties to toric varieties.
\newblock {\em Transform. Groups}, 1(3):215--248, 1996.

\bibitem[HHL08]{HaglundHaimanLoehr08}
Jim Haglund, Mark Haiman, and Nicholas Loehr.
\newblock A combinatorial formula for nonsymmetric {M}acdonald polynomials.
\newblock {\em Amer. J. Math.}, 130(2):359--383, 2008.

\bibitem[KM04]{KnutsonMiller04}
Allen Knutson and Ezra Miller.
\newblock Subword complexes in {C}oxeter groups.
\newblock {\em Adv. Math.}, 184(1):161--176, 2004.

\bibitem[KM05a]{KnutsonMiller05}
Allen Knutson and Ezra Miller.
\newblock Gr\"{o}bner geometry of {S}chubert polynomials.
\newblock {\em Ann. of Math. (2)}, 161(3):1245--1318, 2005.

\bibitem[KM05b]{KoganMiller05}
Mikhail Kogan and Ezra Miller.
\newblock Toric degeneration of {S}chubert varieties and {G}elfand-{T}setlin
  polytopes.
\newblock {\em Adv. Math.}, 193(1):1--17, 2005.

\bibitem[KST12]{KST}
Valentina Kirichenko, Evgeny Smirnov, and Vladlen Timorin.
\newblock Schubert calculus and {G}elfand-{Z}etlin polytopes.
\newblock {\em Uspekhi Mat. Nauk}, 67(4(406)):89--128, 2012.

\bibitem[Las04]{Lascoux04}
Alain Lascoux.
\newblock Schubert \& {G}rothendieck: un bilan bid\'{e}cennal.
\newblock {\em S\'{e}m. Lothar. Combin.}, 50:Art. B50i, 32, 2003/04.

\bibitem[LS90]{LascouxSchutzenberger90}
Alain Lascoux and Marcel-Paul Sch\"{u}tzenberger.
\newblock Keys \& standard bases.
\newblock In {\em Invariant theory and tableaux ({M}inneapolis, {MN}, 1988)},
  volume~19 of {\em IMA Vol. Math. Appl.}, pages 125--144. Springer, New York,
  1990.

\bibitem[Man98]{Manivel98}
Laurent Manivel.
\newblock {\em Fonctions sym\'etriques, polyn\^omes de {S}chubert et lieux de
  d\'eg\'en\'erescence}, volume~3 of {\em Cours Sp\'ecialis\'es [Specialized
  Courses]}.
\newblock Soci\'et\'e Math\'ematique de France, Paris, 1998.

\bibitem[Mas09]{Mason09}
S.~Mason.
\newblock An explicit construction of type {A} {D}emazure atoms.
\newblock {\em J. Algebraic Combin.}, 29(3):295--313, 2009.

\bibitem[PY24]{PanYu23}
Jianping Pan and Tianyi Yu.
\newblock Top-degree components of {G}rothendieck and {L}ascoux polynomials.
\newblock {\em Algebr. Comb.}, 7(1):109--135, 2024.

\bibitem[RY21]{ReinerYong21}
Victor Reiner and Alexander Yong.
\newblock The ``{G}rothendieck to {L}ascoux'' conjecture.
\newblock {\em arXiv preprint arXiv:2102.12399}, 2021.

\bibitem[SY23]{ShimozonoYu23}
Mark Shimozono and Tianyi Yu.
\newblock Grothendieck-to-{L}ascoux expansions.
\newblock {\em Trans. Amer. Math. Soc.}, 376(7):5181--5220, 2023.

\bibitem[Yu23]{Yu21}
Tianyi Yu.
\newblock Set-valued tableaux rule for {L}ascoux polynomials.
\newblock {\em Comb. Theory}, 3(1):Paper No. 13, 31, 2023.

\end{thebibliography}

\end{document}